\documentclass{elsarticle}

\usepackage[a4paper,left=30mm,width=150mm,top=25mm,bottom=25mm]{geometry}

\usepackage{amssymb}
\usepackage{amsfonts}
\usepackage{verbatim}
\usepackage{xcolor}
\usepackage{amsmath}
\usepackage{amsthm}

\usepackage{graphicx}
\usepackage{subfigure}

\usepackage{booktabs} 
\usepackage{adjustbox}
\usepackage{multirow}

\newtheorem{lemma}{Lemma}[section]
\newtheorem{theorem}{Theorem}[section]
\newtheorem*{thm}{Main Theorem}
\newtheorem{corollary}{Corollary}[section]
\newtheorem{proposition}{Proposition}[section]
\newtheorem{remark}{Remark}[section]

\usepackage[all]{xy}

\usepackage{pifont}

\usepackage{lineno,hyperref}
\modulolinenumbers[5]

\journal{}

\biboptions{sort&compress}

\begin{document}

\begin{frontmatter}

\title{Non-convergence and convergence of bounded solutions to semilinear wave equation with dissipative boundary condition\tnoteref{mytitlenote}}

\tnotetext[mytitlenote]{This research was supported by the National Natural Science Foundation of China (12272297).}

\author[mainaddress]{Zhe Jiao\corref{mycorrespondingauthor}}
\cortext[mycorrespondingauthor]{Corresponding author}
\ead{zjiao@nwpu.edu.cn}

\author[mainaddress]{Xiao Li}
\ead{lixiaoo@mail.nwpu.edu.cn}

\address[mainaddress]{School of Mathematics and Statistics, MOE Key Laboratory for Complexity Science in Aerospace, Northwestern Polytechnical University, Xi'an 710129, China}


\begin{abstract}
This paper is concerned with the long-time dynamics of semilinear wave equation subject to dissipative boundary condition.
To do so, we first analyze the set of equilibria, and show it could contain infinitely many elements. 
Second, we show that, for some nonlinear interior sources, the wave equations have solutions that do not stabilize to any single function, while they approach a continuum of such functions.
Finally, if the interior source is a {\L}ojasiewicz-type function, the solution of the wave equation converges to an equilibrium at a rate that depends on the {\L}ojasiewicz exponent, although the set of equilibria is infinite.

\end{abstract}

\begin{keyword}
wave equation \sep dissipative boundary condition \sep critical exponent \sep convergence to an equlibirum
\end{keyword}

\end{frontmatter}


\section{Introduction}
\label{sec:intro}

Let $\Omega$ be a bounded domain in $\mathbb{R}^{d}$, $d = 3$, with a smooth boundary $\Gamma$. The outward normal direction to the boundary is denoted by $\nu$. $f_i$, $i=0, 1$, and $g$ are nonlinear functions.
We consider the following semilinear wave equation
\begin{equation}
	u_{tt} - \Delta u +f_0(x, u) = 0, \quad (t, x)\in\mathbb{R}^{+} \times \Omega \label{wave}
\end{equation}
subject to the dissipative boundary condition
\begin{equation}
	\partial_{\nu}u + f_1(u) + g(u_{t})= 0,  \quad (t, x)\in \mathbb{R}^{+} \times \Gamma \label{dynamic_bdry} 
\end{equation}
and the initial conditions
\begin{equation}
	 u|_{t=0} = u_{0}(x),  \quad  u_{t}|_{t=0} = u_{1}(x), \quad x\in \Omega.\label{initial_data} 
\end{equation}
The investigation of the long-time behavior of solutions to this wave equation has attracted many researchers' attention. Based on the dissipative nature of the system,
\[
	\textrm{energy $\mathcal{E}(u)$ is dissipated by boundary damping $g(\cdot)$}
\]
it can be divided into three categories: energy decay, attractiveness and convergence to equilibria.
In the first category, it is concerned with the problem~(Q1): the decay rates of solutions in the finite energy space as time approaches infinity. For an in-depth investigation of this problem, we direct the readers to seminal works \cite{LT1993, CCP2004, CCL2007} and references therein.
For the study of the second category, there is a series of questions asked on the attractiveness problem (Q2) as follows: (Q2-1) the existence and structure of global attractors;~(Q2-2) the regularity of attractors; (Q2-3) the fractal dimension of attractors. In this direction, we refer to a series of papers \cite{CEL2002, CEL2004, chueshov2009, rodrigues2022, HLY2025}.
Regarding the third category, an interesting problem~(Q3) is whether the global solution will converge to an equilibrium or not when the equilibria are multiple. To be more precise, (Q3) involves three specific questions as follows:
\begin{itemize}
	\item (Q3-1) whether the set of the equilibria contains infinitely many elements; 
	\item (Q3-2) whether some nonlinearities $f_0$ can cause the bounded solution of the wave equation not to converge to any single function;
	\item (Q3-3) whether the bounded solution of the wave equation converge to an equilibrium if $f_0$ satisfies some assumptions.
\end{itemize}
Up to now, there has been limited work on these questions~\cite{CEL2002, chueshov2004, Hao2006, Jiao2018}. It should be noted that in~\cite{CEL2002, chueshov2004} the authors show that the solution converges to an equilibrium if the set of equilibria satisfies some assumption, while in~\cite{Hao2006, Jiao2018} the authors merely addressed question (Q3-3) under other certain constraints. Building on prior research, the aim of our project is to provide comprehensive answers to each of these questions individually.

\subsection{Main results}

In this paper, we focus on the problem (Q3) of the following wave equation
\begin{equation}\label{wave_eq2}
\begin{split}
	u_{tt} - \Delta u +f_0(x, u) &= 0, \quad (t, x)\in\mathbb{R}^{+} \times \Omega,\\
	\partial_{\nu}u + u + g(u_{t})&= 0,  \quad (t, x)\in \mathbb{R}^{+} \times \Gamma,\\
	u|_{t=0} = u_{0}(x), \quad u_{t}|_{t=0} &= u_{1}(x), \quad x\in \Omega. 
\end{split}
\end{equation}
Here, the nonlinear functions $f_0$ and $g$ satisfy the following assumptions.
\begin{itemize}
\item (A1) $f_0$ is $C^2$ in $x \in \bar{\Omega}$. Moreover, $f_0(x, 0) = 0$ and there exists a positive constant $C$ such that
\begin{equation} \label{assumption_critical}
	|f_0^{\prime\prime}(\cdot, s)| \leq C(1+|s|), \quad  \forall s \in \mathbb{R}.
\end{equation}
\item (A2) for all $x \in \bar{\Omega}$,
\begin{equation} \label{assumption_growth_1}
    \liminf_{|s| \rightarrow \infty}\frac{f_{0}(\cdot, s)}{s}\geq -c,
\end{equation}
where $c$ is a constant satisfying $c<\mu_{0}$, and $\mu_{0}$ is the best constant in the following Poincar\'e-type inequality
\begin{eqnarray*}
	\int_{\Omega}|\nabla u|^2\mathrm{d}x + \int_{\Gamma}|u|^2\mathrm{d}S \geq \mu_{0}\int_{\Omega}|u|^2\mathrm{d}x.
\end{eqnarray*}
\item (A3) $g\in C^{1}(\mathbb{R})$ with $g(0)=0$, $g$ is monotonically increasing, and there exist two positive constants $m_1$ and $m_2$ such that
\begin{equation} \label{assumption_g}
	0<m_{1} \leq g^{\prime}(s) \leq m_{2}< \infty, \quad s\in\mathbb{R}.
\end{equation}
\end{itemize}

In what follows we state a simplified version of the main result of this paper (see Theorem~\ref{equilibria:infinity},~\ref{thm:main1} and~\ref{thm:main2} for exact statements).
\begin{thm}
Under the assumptions {\rm{(A1)}}, {\rm{(A2)}} and {\rm{(A3)}} on the functions $f_0$ and $g$, we have the following properties.
\begin{itemize}
\item For some nonlinearities $f_0$, the corresponding stationary problem of equation~\eqref{wave_eq2} has infinitely many nontrivial solutions.
\item There exists a $C^{\infty}$ function $f_0: \bar{\Omega}\times\mathbb{R}\rightarrow\mathbb{R}$ such that the $\omega-$limit set of the bounded solution to~\eqref{wave_eq2} is a continuum in $H^1(\Omega)$ homeomorphic to the unit circle $S^1$.  
\item Suppose the domain $\Omega$ is quasi-star-shaped. If $f_0$ also satisfies the {\L}ojasiewicz condition, then the bounded solution to equation~\eqref{wave_eq2} converges to an equilibrium in the topology of the finite energy space as time goes to infinity. Moreover, the convergence rate depends on the {\L}ojasiewicz exponent.
\end{itemize}
\end{thm}

\subsection{Related work}
\label{sec:literature}

Before giving the detailed proof of our main results, let us give a short survey of the literature relevant to this initial-boundary value problem for the semilinear wave equation.

To study the asymptotic behavior of the solution, the existence and uniqueness of both local and global solutions to the wave equation should be investigated first. There is a considerable number of papers dealing with the well-posedness problem \cite{LT1993, CCP2004, CCL2007, CEL2002, Vitillaro_2002, vitillaro2002global, Cavalcanti2006}. 
In the case that the boundary source $f_1$ is linear, or absent, the author \cite{Vitillaro_2002} makes use of the potential well theory to investigate the existence of global solutions as well as the blow-up of weak solutions in finite time. Meanwhile, in~\cite{CCP2004} the authors combine the potential well theory with the Faedo-Galerkin procedure to obtain the existence of global solutions.
Another method to study the well-posedness problem is based on the nonlinear semigroup theory, as presented in \cite{CEL2002} and \cite{Cavalcanti2006}. 
As to the presence of nonlinear boundary source $f_1$, it leads to the fact that Lopatinski condition~\cite{lasiecka1990sharp, tataru1998regularity} fails for Neumann hyperbolic map. Thanks to the \textit{dissipation mechanism}, even linear damping changes problem to the one where Lopatinski condition is satisfied~\cite{Lasiecka_Triggiani_2000}. This is referred to as the smoothing effect. With the aid of this effect, the authors in~\cite{CCL2007} and~\cite{rodrigues2022} apply the same technique as in~\cite{LT1993} to address the issue of well-posedness. An excellent literature review of this context can be found in~\cite{CCL2007}.

Suppose that for any given initial datum, the wave equation admits a unique global solution in a certain function space. 
Naturally, we can ask the problem (Q1), that is, how does the energy functional defined by the global solution evolves? We provide a brief outline of the research in this direction.
Beginning with a linear wave equation, the authors in~\cite{bardos1992sharp,martinez1999} provide a deep insight into the relationship between energy decay and linear boundary damping. In particular, the paper~\cite{bardos1992sharp} gives sharp sufficient conditions for stabilization of the wave equation from the boundary.
The papers~\cite{zuazua1990uniform, martinez1999new} prove that the energy of the wave equation decays to zero with an explicit decay rate under some geometrically restricted conditions on the nonlinear boundary damping.
Considering the interaction of nonlinear interior source term and nonlienar boundary damping, the authors in~\cite{CCP2004, Cavalcanti2006} prove the uniform decay rates for the wave equation.
For the general case, we refer to the papers~\cite{LT1993, CCL2007}, where a semilinear model of the wave equation with nonlinear interior source and nonlinear boundary damping- source interaction. Moreover, the authors in~\cite{CCL2007} establish effective optimal decay rates of the energy for the semilinear wave equation.

The analysis for establishing decay rates is a significant preliminary step in studying the attractiveness problem (Q2) of the semilinear wave equation. In order to prove the existence of attractor, there are two frameworks illustrated in the following diagram.
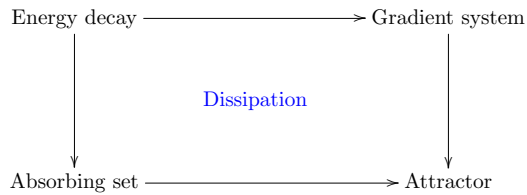
\begin{equation*}
\scalebox{0.8}{
\xymatrix{
\textrm{Energy decay}  \ar@{->}[0,2]  \ar@{->}[2, 0]  &  & \textrm{Gradient system}\ar@{->}[2,0]\\
 & {\color{blue}\textrm{Dissipation}} & \\
\textrm{Absorbing set} \ar@{->}[0,2] & &\textrm{Attractor}
}
}
\end{equation*}
It can be seen that the \textit{dissipation mechanism} is the core to address (Q1) and (Q2).
In this context, it is important to mention the seminal works published by Lasiecka's group~\cite{CEL2002, chueshov2004, CEL2004, chueshov2009}. Notice that the research in~\cite{CEL2002, chueshov2004, CEL2004} deals with the wave equation~\eqref{wave_eq2}, in which the nonlinear damping acts on the entire boundary but is unquantized at the origin. In~\cite{CEL2002}, the authors show the first result that the equation~\eqref{wave_eq2} possesses a global and compact attractor in finite energy space, which coincides with the unstable manifold emanating from stationary solutions. The strategy in~\cite{CEL2002} relies on the proof that the dynamical system associated with the wave equation possesses an absorbing set and is asymptotically smooth~\cite{chueshov2015dynamics}.
These results established in \cite{CEL2002} form the basis for the subsequent results in \cite{chueshov2004, CEL2004}, where the authors show the fractal dimension of the global attractor is finite only in a subcritical case of interior source. 
Regarding the case of localized boundary damping, the authors provide answers to the three questions of (Q2) in~\cite{chueshov2009}.
Recently, the authors in~\cite{rodrigues2022} show the corresponding dynamical system is of gradient-type and asymptotically smooth, and then extend the results in~\cite{CEL2002} to the general case of wave equation subject to nonlinear boundary damping and nonlinear interior and boundary sources with critical exponents. Later,~\cite{HLY2025} complements these results of \cite{rodrigues2022} with estimates of finite dimension and regularity of global attractors.

Assume that the global solution of the wave equation~\eqref{wave_eq2} is also bounded in a certain norm. For the asymptotic behavior of the solution, there are two possibilities (Q3-2) and (Q3-3) as time goes to infinity. To our best knowledge, the problem (Q3-2) has not been studied in earlier works, while limited work~\cite{CEL2002, CEL2004, chueshov2004, Hao2006, Jiao2018} has been done for (Q3-3) as mentioned before. However, the convergence results in these few papers hold true only when the damping acts on the entire boundary and additional requirements are imposed on the equation~\eqref{wave_eq2}.
If the set of equilibria is finite, the convergence of the bounded solution to an equilibrium is proved in~\cite{CEL2002, CEL2004, chueshov2004}. We need to point out in Theorem~\ref{equilibria:infinity} of our work that this finite assumption on the set of equilibria is not always true. In the case of the presence of linear interior and boundary dampings, the convergence result is shown in~\cite{Hao2006}. Given initial data small enough, local convergence result that the bounded solutions stabilize to zero equilibrium is provided in~\cite{Jiao2018}. The convergence result of this paper (see Theorem~\ref{thm:main2}) is valid in the absence of the interior damping and without any assumptions imposed on the initial data.
 
Given the above discussion, the main novelty of our contribution to solving the problem (Q3) can be summarized as follows: 
(i) We impose some conditions on the interior source that cause the set of equilibria to be infinite, and construct an interesting example satisfying these conditions, thereby answering (Q3-1);
(ii) We affirmatively answer (Q3-2) by proving the $\omega$-limit set of the bounded solution to the wave equation with some smooth interior source is a continuum in $H^1$ homeomorphic to the unit circle. While this type of result has been proved for damped wave equation with Dirichlet boundary condition~\cite{jendoubi2003non}, this paper provides the first result for the case of dissipative boundary condition;
(iii) We show that the convergence result holds under the assumption that the interior source is a {\L}ojasiewicz-type function, thereby affirmatively answering (Q3-3).
The detailed comparison of this paper with related studies is collected in Table~\ref{overview}.

\begin{table}[htbp]
\begin{adjustbox}{center}
\scalebox{0.8}{
\begin{tabular}{ccccccccccc}
\toprule
\multirow{2}{*}{Paper} &  \multirow{2}{*}{$f_0$} &  \multirow{2}{*}{$f_1$} &  \multirow{2}{*}{$g$} & Energy decay &  \multicolumn{3}{c}{Attractiveness} &  \multicolumn{3}{c}{Convergence}  \\ \cmidrule{5-11}
& & & & Q1 & Q2-1 &Q2-2 &Q2-3 &Q3-1& Q3-2 &Q3-3\\
\midrule
\cite{LT1993} & subcritical  & \multirow{2}{*}{subcritical} & \multirow{2}{*}{nonlinear} & \ding{51}  & & & & & &\\
\cite{CCL2007} & (sub)critical &  &  & \ding{51}  & & & & & &\\
\midrule
\cite{CEL2002} & critical  & \multirow{6}{*}{linear}  &   \multirow{4}{*}{nonlinear$^\S$} & &\ding{51} & \ding{51} & & & &\\  \cmidrule{2-2}
\cite{CEL2004} & \multirow{2}{*}{subcritical} &  & &  & & &  \ding{51} & &   & \ding{51}$^\dag$ \\
\cite{chueshov2004} &  &  & &  & & &  \ding{51} & &   & \ding{51}$^\dag$ \\  \cmidrule{2-2}
\cite{Jiao2018} & \multirow{2}{*}{critical} &  & &  & & & & & &\ding{51}$^\sharp$ \\ \cmidrule{4-4}
\cite{Hao2006} &  &  & linear &  & & & & & & \ding{51}$^\natural$  \\
\midrule
\cite{chueshov2009} & \multirow{2}{*}{critical} &  \multirow{2}{*}{linear} & nonlinear$^\flat$ &  & \ding{51} & \ding{51} & \ding{51} & &  & \\
{\color{blue}This work}  &  &  & nonlinear &  &   &  &  &  \ding{51} & \ding{51} & \ding{51}$^\ddag$\\
\midrule
\cite{rodrigues2022} &  \multirow{3}{*}{critical}   &  \multirow{3}{*}{critical} &  \multirow{3}{*}{nonlinear} &  &  \ding{51} & & &\\
\cite{HLY2025} &  &  & &  & & \ding{51} &  \ding{51} & & &\\
$/$  & &  & & & & & & & & {\color{red}open} \\
\bottomrule
\end{tabular}
}
\end{adjustbox}
\caption{Overview of the investigation about the long time behaviors of solutions to the wave equation~\eqref{wave} with initial-boundary value~\eqref{dynamic_bdry} and~\eqref{initial_data}. The superscript 
$^\S$ means that the nonlinear damping is unquantized at the origin;
$^\flat$ the damping acts on a subset of the boundary;
$^\dag$ the convergence property holds under finite assumption on $\mathcal{N}$;
$^\natural$ the convergence property holds under;
$^\sharp$ the convergence property holds under; 
$^\ddag$ the convergence property holds under the assumption that $f_0$ satisfies {\L}ojasiewicz condition.}
\label{overview}
\end{table}

\paragraph{Organization of the paper}
The remaining part of this paper is organized as follows.
Section~\ref{sec:evolution} briefly recalls the related well-posedness results of the equation~\eqref{wave_eq2}. 
In Section~\ref{sec:stationary}, we study the existence of multiple of solutions for a nonlinear elliptic equation with Robin boundary condition, which is the stationary problem of the equation~\eqref{wave_eq2}. In Section~\ref{sec:nonconvergence} we derive the non-convergence result of bounded solutions to the equation~\eqref{wave_eq2} via the centre-manifold reduction. Section~\ref{sec:convergence} discusses the {\L}ojasiewicz-type landscape, and gives a detailed proof of the convergence result for the equation~\eqref{wave_eq2}. Section~\ref{sec:conclusion} provides some conclusions.

\paragraph{Notation}
$C^{k}(\cdot)$ is the space of $k-$times continuously differentiable functions on the indicated set. $C^{\infty}(\cdot)$ is the space of infinitely differentiable functions on the indicated set. $C^k_{\mathrm{b}}(\cdot)$ is the space of functions whose derivative up to order $k$ are continuous and bounded on the indicated set, and the corresponding norm is given by the supremum of all these derivatives.

 $f^{\prime}(\cdot, s)$, $g^{\prime}(s)$ are the first-order derivative of $f$ and $g$ with respect to $s$, and $f^{\prime\prime}(\cdot, s)$ is the second-order derivative of $f$ with respect to $s$. 

The Sobolev spaces $W^{m, 2}$ are denoted by $H^m(\Omega)$ with norm $\|\cdot\|_{H^m}$, and especially $H^1(\Omega)$ is equipped with the norm
\[
	\|u\|_{H^1} = \left(\int_{\Omega}|\nabla u|^2\mathrm{d}x + \int_{\Gamma}|u|^2\mathrm{d}S \right)^{\frac{1}{2}}.
\]
The norm in $L^2(\Omega)$ is simply denoted by $\|\cdot\|$.

Let $\mathcal{H}:=H^{1}(\Omega)\times L^{2}(\Omega)$ be the finite energy space with the norm given by 
\[
	\frac{1}{2}\|(u, u_t)^{\top}\|^2_{\mathcal{H}} = \frac{1}{2}\left(\int_{\Omega}|\nabla u|^2\mathrm{d}x + \int_{\Gamma}|u|^2\mathrm{d}S + \int_{\Omega}|u_t|^2\mathrm{d}x\right).
\]

$c$, $c_i$, $C$, $C_i$, and $C(\cdots)$ denote positive numbers, which may depend on the quantities listed in brackets.

\section{Well-posedness}
\label{sec:evolution}
For the better readability and completeness, this section briefly recalls the well-posedness results for the wave equation~\eqref{wave_eq2}, which have been derived via the nonlinear semigroup theory in~\cite[Theorem 2.1]{CEL2002} and~\cite[Theorem 1]{rodrigues2022}.  
Let us define the energy function
\[
	\mathcal{E}(u) = \frac{1}{2}\|(u, u_t)^{\top}\|^2_{\mathcal{H}} + \int_{\Omega}F_0(x, u)\mathrm{d}x
\]
with $F_0(x, u) = \int_{0}^{u}f_0(x, s)\mathrm{d}s$, and introduce the space
\begin{equation*}
\mathcal{D} :=
 \left\{
\left(
\begin{array}{c}
u \\
v
\end{array}
\right)
\in H^2(\Omega)\times H^1(\Omega): \partial_{\nu} u + u +  g(v) = 0\quad  \textrm{on $\Gamma$}
\right\},
\end{equation*}
which clearly is a closed subspace of $H^{2}(\Omega)\times H^{1}(\Omega)$.
Now, we state the well-posedness result of problem~\eqref{wave_eq2}, and provide a sketch of the proof in the following.

\begin{lemma}\label{wellposedness}
{\rm(Existence and regularity of solutions)}
\begin{itemize}
\item Suppose that the initial data $(u_0, u_1)^{\top}$ belongs to $\mathcal{H}$. Then there exists a unique weak solution $(u, u_t)^\top \in C([0, \infty); \mathcal{H})$ to problem~\eqref{wave_eq2} such that
\begin{itemize}
\item $\partial_{\nu}u \in L^{2}_{\mathrm{loc}}([0, \infty)\times\Gamma)$, $u_t \in L^{2}_{\mathrm{loc}}([0, \infty)\times\Gamma)$, and $\partial_{\nu}u + u + g(u_t) = 0$ on $\Gamma$,
\item weak solutions satisfy the following energy identity
	\[
		\mathcal{E}(u(t))  + \int_{0}^{t}\langle  g(u_s(s)), u_s(s)\rangle \mathrm{d}s = \mathcal{E}(u(0)),
	\]
\item and weak solutions satisfy the following variational equality with derivatives understood in the sense of distribution
\[
	\frac{d}{dt}(u_t, \phi) + (\nabla u, \nabla\phi) + \langle  g(u_t) + u, \phi\rangle + (f(x, u), \phi) = 0, \quad \forall \phi\in H^{1}(\Omega).
\]	
\end{itemize}
\item Furthermore, if the initial data $(u_0, u_1)^{\top}\in\mathcal{D}$, then the corresponding weak solution is classical and possesses the regularity properties
\[
	(u, u_t)^\top \in C([0, \infty); \mathcal{D})\cap C^1([0, \infty); \mathcal{H}), \quad u_{tt}\in C([0, \infty); L^2({\Omega})).
\] 
\end{itemize}
\end{lemma}
\begin{proof}
First, we need to reformulate the origin problem~\eqref{wave_eq2} as an abstract first-order evolution equation. To do this, we introduce the Robin-Laplacian operator $\triangle_R: L^{2}(\Omega) \rightarrow L^2(\Omega)$ with the domain
\[
	D(\triangle_R) = \left\{ u\in H^2(\Omega): \partial_{\nu} u + u = 0\quad  \textrm{on $\Gamma$}\right\}.
\]
This densely defined operator is injective and self-adjoint. Moreover, it can be extend to a continuous operator,  which is still denoted by $\triangle_R$. This extension is the duality map from $ H^{1}(\Omega)$ to $ H^1(\Omega)^{\prime}$ given by
\[
	(-\triangle_R u, v) = \int_{\Omega}\nabla u\cdot \nabla v \mathrm{d}x + \int_{\Gamma}uv\mathrm{d}S
\]
for any $v\in H^{1}(\Omega)$.

Next we introduce the Robin map $R: H^{s}(\Gamma)\rightarrow H^{s+\frac{3}{2}}(\Omega)$, which is defined as
\[
	Rp = q \Longleftrightarrow \triangle q = 0\quad \textrm{in $\Omega$, and }  \partial_{\nu} q + q = p \quad \textrm{on $\Gamma$}.
\] 
By the elliptic theory in \cite{lions2012non} we know $R$ is continuously for all $s\in\mathbb{ R}$.
Given $p\in H^{-\frac{1}{2}}(\Gamma)$ and $v\in D(\triangle_{R})$ we have
\begin{equation*}
\begin{split}
	(\triangle_{R}v, Rp)  = (\triangle_{R}v, q) &= - \int_{\Omega}\nabla v\cdot \nabla q \mathrm{d}x + \int_{\Gamma}\partial_{\nu}v q\mathrm{d}S\\
	&= - \int_{\Gamma} v \partial_{\nu}q \mathrm{d}S + \int_{\Omega}v\triangle q \mathrm{d}x + \int_{\Gamma}\partial_{\nu}v q\mathrm{d}S \\
	&=- \int_{\Gamma} v \partial_{\nu}q \mathrm{d}S - \int_{\Gamma}v q\mathrm{d}S = -\langle v, p \rangle.
\end{split}
\end{equation*}
which implies
\[
	\langle R^{\ast}\triangle_{R}v, p \rangle = (\triangle_{R}v, Rp) = -\langle v, p \rangle.
\]
Due to the fact that $D(\triangle_R)$ is dense in $H^1(\Omega)$, the adjoint of the Robin map satisfies
\[
	R^{\ast}\triangle_{R}v = -T v = -v|_{\Gamma}, \quad \forall v\in H^1(\Omega),
\]
where $T$ is the trace operator.
Moreover, the fact that $\triangle_{R}$ is self-adjoint implies
\[
	\langle -Tv, p \rangle= \langle R^{\ast}\triangle_{R}v, p \rangle = (\triangle_{R}v, Rp)  = \langle v,\triangle_{R}Rp \rangle.
\]
Then we have $ T^\ast= - \triangle_{R}R$.

Based on the preliminary work above, we turn to introduce a nonlinear operator $\mathcal{A}$ with the domain
 \begin{equation*}
D(\mathcal{A}) = \left\{ 
\left(
\begin{array}{c}
u \\
v
\end{array}
\right)
\in H^1(\Omega)\times H^1(\Omega): u+R\left( g(T v)\right)\in D(\triangle_R)
\right\}
\end{equation*}
by setting 
 \begin{equation}\label{mono_operator}
\mathcal{A }=
\left(
\begin{array}{cc}
0 & -I\\
-\triangle_R & T^\ast \left( g(T\cdot)\right)
\end{array}
\right)
 \Longleftrightarrow
\mathcal{A}
\left(
\begin{array}{c}
u \\
v
\end{array}
\right)
=
\left(
\begin{array}{c}
-v\\
-\triangle_R(u+R\left( g(T v)\right)) 
\end{array}
\right).
\end{equation}
Since we have $D(\mathcal{A})\subseteq H^2(\Omega)\times H^1(\Omega)$, the mapping $v\mapsto  g(Tv)$ is continuous from $H^1(\Omega)$ to $H^{\frac{1}{2}}(\Gamma)$, and
\[
	0= \partial_{\nu}[u+R\left( g(T v)\right)] + u+R\left( g(T v)\right)= \partial_{\nu} u + u + g(v) \quad \textrm{on $\Gamma$},
\]
then it deduces that $D(\mathcal{A}) =\mathcal{D}$. 

Now, we can rewrite the wave equation~\eqref{wave_eq2} as the following operator theoretic formulation 
\begin{equation}\label{evolution_eq}
\partial_t\left(
\begin{array}{c}
u \\
u_t
\end{array}
\right)
+\mathcal{A}
\left(
\begin{array}{c}
u \\
u_t
\end{array}
\right)
+
\left(
\begin{array}{c}
0 \\
f_0(x, u)
\end{array}
\right)
=0.
\end{equation}
In a similar manner as in~\cite{CEL2002}, the conclusion of this lemma can be derived by using the theory of maximal monotone operators~\cite{Barbu2010} and an approximation method~\cite{LT1993}.
\end{proof}

\section{Infinitely many equilibria}
\label{sec:stationary}

An equilibrium $\varphi\in H^{2}(\Omega)$ to problem~\eqref{wave_eq2} is a classical solution to the following nonlinear elliptic equation with Robin boundary condition.
\begin{equation}\label{eq_stationary}
\left
    \{
        \begin{array}{rl}
            - \Delta \varphi +f_0(x, \varphi) = 0 &\textrm{in $ \Omega$,} \\
                 \partial_{\nu} \varphi + \varphi= 0 &\textrm{on $\Gamma$.}
        \end{array}
\right.
\end{equation}
From Lemma 3.1 in~\cite{Hao2006} we can deduce that an equilibrium $\varphi$ is also a critical point of the following functional in $H^{1}(\Omega)$
\[
	E(u) = \frac{1}{2}\int_{\Omega}|\nabla u|^2 \mathrm{d}x +  \frac{1}{2}\int_{\Gamma}|u|^2\mathrm{d}S + \int_{\Omega}F_0(x, u)\mathrm{d}x
\]
with $F_0(x, u) = \int_{0}^{u}f(x, s)\mathrm{d}s$.
The set of all the equilibria is denoted by $\mathcal{N}$. 
It is obvious that the structure of the set $\mathcal{N}$ depends on the nonlinearity $f_0$.

In the following, we are about to search the existence of infinitely many solutions for the stationary problem~\eqref{eq_stationary}.

\begin{theorem}[Infinitely many solutions] \label{equilibria:infinity}
Let $f_0(x, s) = \lambda(x)s - \hat{f}(x, s)$ and $\lambda(x)\in C(\bar{\Omega})$. Assume $\hat{f}(x, s)$ satisfies the following conditions.
\begin{itemize}
	\item[{\rm(1)}] For $2<q<2^{\ast}=\frac{2d}{d-2}$, 
		\[
			|\hat{f}(x, s)|\leqslant C\left( 1+ |s|^{q-1} \right).
		\]
	\item[{\rm(2)}] Given any $R>0$, there exists a constant $\kappa > 2$ such that
		\[
			s\hat{f}(x, s)-\kappa\hat{F}(s) \geqslant 0, \quad |s|\leqslant R
		\]
		with with $\hat{F}_0(s) = \int_{0}^{s}\hat{f}(t)\mathrm{d}t$, and
		\[
			\lim_{|s|\rightarrow\infty}\frac{s\hat{f}(x, s)-\kappa\hat{F}(s)}{s^2} = 1-\frac{\kappa}{2}.
		\]
	\item[{\rm(3)}] $\hat{f}(x, s)$ is odd, that is, $-\hat{f}(s)=\hat{f}(-s)$.
\end{itemize}
Then if $\lambda(x) > \frac{\kappa}{\kappa-2}$, the problem~\eqref{eq_stationary} has a sequence of solutions $\varphi_n\in H^1(\Omega)$ such that $E(\psi_n)\rightarrow\infty$ as $n$ goes to infinity.
\end{theorem}
It can be observed that the second condition imposed on \(\hat{f}\) is weaker than the corresponding condition in Theorem 1.1 of~\cite{Zhao2008}. Consequently, this theorem generalizes Theorem 1.1 in~\cite{Zhao2008}. To prove this theorem, we employ the Fountain Theorem proposed in~\cite{willem2012minimax}.
\begin{proof}[Proof of Theorem~\ref{equilibria:infinity}]
The subsequent proof involves two steps.

\textbf{Step 1}. Suppose that $\{\varphi_n\}\subset H^{1}(\Omega)$, for every $C_1>0$,
\[
	E(\varphi_n)\rightarrow C_1,\quad E^{\prime}(\varphi_n)\xrightarrow[]{H^{1}(\Omega)^{\ast}} 0, \quad n\rightarrow\infty.
\]
Here, $E^{\prime}(\cdot)$ is the first-order variation of the functional $E(\cdot)$.
We will show that $\{\varphi_n\}$ has a convergent subsequence in $H^{1}(\Omega)$.

From the second condition, for any $\epsilon > 0$ there exists $N>0$ such that 
\[
	\left|s\hat{f}(x, s)-\kappa\hat{F}(s) - \left( 1-\frac{\kappa}{2} \right)\right| \leqslant \epsilon, \quad |s| \geqslant N.
\] 
Thus, we obtain
\begin{equation}\label{thm1pf_1}
\begin{split}
	\frac{1}{\kappa}s\hat{f}(x, s)-\hat{F}(s) &\geqslant 0, \quad  |s| \leqslant R, \\
	\frac{1}{\kappa}s\hat{f}(x, s)-\hat{F}(s) &\geqslant \frac{1}{\kappa}-\frac{1}{2}-\frac{\epsilon}{\kappa}, \quad  |s| \geqslant N.
\end{split}
\end{equation}

Since we know
\begin{equation*}
\begin{split}
	&E(\varphi_n) - \frac{1}{\kappa}\langle E^{\prime}(\varphi_n), \varphi_n \rangle\\
	&=\left( \frac{1}{2} - \frac{1}{\kappa}\right)\left\{\int_{\Omega}\left[ |\nabla\varphi_n|^2 + \lambda(x)|\varphi_n|^2\right]\mathrm{d}x+\int_{\Gamma}|\varphi_n|^2\mathrm{d}S\right\}+ \int_{\Omega}\left[\frac{1}{\kappa}\hat{f}(x, \varphi_n)\varphi_n - \hat{F}(x, \varphi_n)\right]\mathrm{d}x,
\end{split}
\end{equation*}
where $E^{\prime}(\cdot)$ is the first derivative of $E$, then from~\eqref{thm1pf_1} we get
\begin{equation*}
	E(\varphi_n) - \frac{1}{\kappa}\langle E^{\prime}(\varphi_n), \varphi_n \rangle\geqslant\left( \frac{1}{2} - \frac{1}{\kappa}\right)\left\{\int_{\Omega}\left[ |\nabla\varphi_n|^2 + \lambda(x)|\varphi_n|^2\right]\mathrm{d}x+\int_{\Gamma}|\varphi_n|^2\mathrm{d}S\right\}
\end{equation*}
for $ |s| \leqslant N$, or
\begin{equation*}
\begin{split}
	&E(\varphi_n) - \frac{1}{\kappa}\langle E^{\prime}(\varphi_n), \varphi_n \rangle\\
	&\geqslant\left( \frac{1}{2} - \frac{1}{\kappa}\right)\left\{\int_{\Omega}|\nabla\varphi_n|^2 dx+\int_{\Gamma}|\varphi_n|^2\mathrm{d}S\right\}+ \left[(\lambda(x)-1)\left(\frac{1}{2} - \frac{1}{\kappa}\right) - \frac{\epsilon}{\kappa}\right]\int_{\Omega}|\varphi_n|^2\mathrm{d}x
\end{split}
\end{equation*}
for $ |s| \geqslant N$.

Due to $\lambda(x) > \frac{\kappa}{\kappa-2}$, then choosing $\epsilon \ll 1$, we have
\begin{equation}\label{thm1pf_2}
	C_1 +1 \geqslant E(\varphi_n) - \frac{1}{\kappa}\langle E^{\prime}(\varphi_n), \varphi_n \rangle\geqslant \left( \frac{1}{2} - \frac{1}{\kappa}\right)\|\varphi_n\|^2_{H^1}
\end{equation}
for $n$ large enough. By~\eqref{thm1pf_2} we obtain the boundedness of $\{\varphi_n\}$ in $H^{1}(\Omega)$, which implies that there exists a subsequence, which is still denoted by $\{\varphi_n\}$, and $\varphi\in H^1(\Omega)$ such that $\varphi_n$ converges weakly to $\varphi$ in $H^1(\Omega)$ as $n\rightarrow\infty$.

Let us introduce three linear operators
\[
	\langle A(\varphi), \phi \rangle = \int_{\Omega}\left[ \nabla\varphi \cdot\nabla\phi + \lambda(x)\varphi\phi \right]\mathrm{d}x,
\]
\[
	\langle B(\varphi), \phi \rangle = \int_{\Gamma}\varphi\phi \mathrm{d}S , \quad \langle G(\varphi), \phi \rangle = \int_{\Omega}\hat{f}(x, \varphi)\phi \mathrm{d}x.
\]
From~\cite{martinez2003weak} we know $A$ is continuously invertible, $B$ and $C$ are continuous and compact.
Then the first-order variation $E^{\prime}(\varphi_n) =  A(\varphi_n) + B(\varphi_n) - G(\varphi_n)$.
By the assumption $E^{\prime}(\varphi_n) \rightarrow 0$, we have
\[
	\varphi_n\rightarrow -A^{-1}\left( B(\varphi) - G(\varphi)\right)
\]
in $H^1(\Omega)$ as $n\rightarrow\infty$. Therefore, $\varphi_n\rightarrow \varphi$ in $H^{1}(\Omega)$ as $n\rightarrow\infty$. 

\textbf{Step 2}. Since $H^1(\Omega)$ is a reflexive and separable Banach space, then there are basis $e_j\in H^1(\Omega)$ and $e^{\ast}_j\in H^1(\Omega)^{\ast}$ satisfying
\[
	\langle e_i, e_j^{\ast} \rangle = 1,\quad i=j; \quad \langle e_i, e_j^{\ast} \rangle = 0,\quad i\neq j
\]
such that
\[
	H^1(\Omega) = \overline{\mathrm{span}\{e_j: j=1, 2, \cdots\}},\quad H^1(\Omega)^{\ast} = \overline{\mathrm{span}\{e_j^{\ast}: j=1, 2, \cdots\}}.
\]
Let us write $\mathcal{X}_j:=\mathrm{span}\{e_j\}$,   $\mathcal{Y}_k:= \oplus_{j=1}^{k}\mathcal{X}_j$ and $\mathcal{Z}_k:= \overline{\oplus_{j=k}^{\infty}\mathcal{X}_j}$.

If $\varphi\in \mathcal{Y}_k$ and $\|\varphi\|_{H^1(\Omega)} =c_k  > 0$, then it deduces from second and third conditions that there exists $C_2$ such that
\[
	\hat{F}_0(x, \varphi)\geqslant C_2|\varphi|^{\kappa},
\]
which implies
\begin{equation}\label{thm1pf_3}
\begin{split}
	E(\varphi)&=\frac{1}{2}\int_{\Omega}\left[ |\nabla \varphi|^2 + \lambda(x)|\varphi|^2\right] \mathrm{d}x +  \frac{1}{2}\int_{\Gamma}|\varphi|^2\mathrm{d}S - \int_{\Omega}\hat{F}_0(x, \varphi)\mathrm{d}x\\
	&\leqslant \frac{1}{2} \|\varphi\|^2_{H^1} + C_3 \|\varphi\|^2 -C_2\|\varphi\|^{\kappa}_{L^{\kappa}(\Omega)}.
\end{split}
\end{equation}
The constant $C_3$ above  is derived from the boundedness of $\lambda(x)$ in $\bar{\Omega}$. 
Since on the finite dimensional space $\mathcal{Y}_k$ all norms are equivalent, it implies from \eqref{thm1pf_3} and $\kappa > 2$ that 
\begin{equation}\label{thm1pf_4}
	\max_{\substack{\varphi\in \mathcal{Y}_k \\ \|\varphi\|_{H^1(\Omega)} =c_k}}E(\varphi) \leqslant 0
\end{equation}
for $c_k$ large enough.

By the first condition we have
\[
	\hat{F}_0(x, \varphi) \leqslant C_4\left(1+|\varphi|^q \right).
\]
Then the following estimate holds
\begin{equation}\label{thm1pf_5}
\begin{split}
	E(\varphi)&=\frac{1}{2}\int_{\Omega}\left[ |\nabla \varphi|^2 + \lambda(x)|\varphi|^2\right] \mathrm{d}x +  \frac{1}{2}\int_{\Gamma}|\varphi|^2\mathrm{d}S - \int_{\Omega}\hat{F}_0(x, \varphi)\mathrm{d}x\\
	&\geqslant \frac{1}{2} \|\varphi\|^2_{H^1} + \frac{\kappa}{2(\kappa-2)} \|\varphi\|^2 -C_4\|\varphi\|^{q}_{L^{q}(\Omega)} -C_4|\Omega|.
\end{split}
\end{equation}

Let us define a parameter $\alpha_k$ by
\[	
	\alpha_k:=\sup_{\substack{\varphi\in \mathcal{Z}_k \\ \|\varphi\|_{H^1(\Omega)}=1}}\|\varphi\|_{L^{q}(\Omega)}.
\]
By Lemma 3.5 in~\cite{Zhao2008}, we get $\alpha_k\rightarrow 0$ as $k\rightarrow\infty$. Thus, on $\mathcal{Z}_k$, it implies from~\eqref{thm1pf_5} that
\begin{equation}\label{thm1pf_6}
\begin{split}
	E(\varphi) &\geqslant \frac{1}{2} \|\varphi\|^2_{H^1} + \frac{\kappa}{2(\kappa-2)} \|\varphi\|^2 -C_4\|\varphi\|^{q}_{L^{q}(\Omega)} -C_4|\Omega|\\
	&\geqslant \frac{1}{2} \|\varphi\|^2_{H^1} - C_4 \alpha^q_k \|\varphi\|^{q}_{H^{1}(\Omega)} -C_4|\Omega|.
\end{split}
\end{equation}
If $\|\varphi\|_{H^1} =\beta_k$, then \eqref{thm1pf_6} gives
\begin{equation}\label{thm1pf_7}
\begin{split}
	E(\varphi) &\geqslant \frac{1}{2}\beta_k^2 - C_4 \alpha^q_k \beta_k^q -C_4|\Omega|\\
	&=\frac{1}{2}\left(C_4 q\alpha_k^q\right)^{\frac{2}{2-q}}- C_4 \alpha^q_k \left(C_4 q\alpha_k^q\right)^{\frac{q}{2-q}} -C_4|\Omega| \\
	&=\frac{1}{2}\left(C_4 q\alpha_k^q\right)^{\frac{2}{2-q}}- \frac{1}{q}\left(C_4 q\alpha_k^q\right)^{\frac{2}{2-q}} -C_4|\Omega|\\
	&=\frac{q-2}{2q}\left(C_4 q\right)^{\frac{2}{2-q}} \alpha_k^{\frac{2q}{2-q}}- C_4|\Omega|
\end{split}
\end{equation}
by choosing $\beta_k = \left(C_4 q\alpha_k^q\right)^{\frac{1}{2-q}}$. Due to $2<q$,  we obtain
\begin{equation}\label{thm1pf_8}
	\inf_{\substack{\varphi\in \mathcal{Z}_k \\ \|\varphi\|_{H^1(\Omega)} =\beta_k}}E(\varphi) \rightarrow \infty
\end{equation}
as $k\rightarrow\infty$

\textbf{Step 3}. Based on the result in Step 1, \eqref{thm1pf_4} and \eqref{thm1pf_8}, it deduces from the Fountain theorem in~\cite{willem2012minimax} that the elliptic equation~\eqref{eq_stationary} has a sequence of solutions $\varphi_n\in H^1(\Omega)$ such that $E(\varphi_n)\rightarrow\infty$ as $n$ goes to infinity.
\end{proof}

Now, we begin to construct a concrete example of $f_0$ that satisfies assumptions~(A1) and~(A2), as well as the three conditions in Theorem~\ref{equilibria:infinity}.

\begin{itemize}
\item \textbf{Example}. If $s\neq0$, we set
	\begin{equation}\label{example}
		f_0(x, s) = \lambda(x)s - se^{-\frac{1}{s^2}}
	\end{equation}
	with $\lambda(x)>1$. Moreover, when $s=0$, we set $f_0(x, 0) = 0$, $f^{\prime}_{0}(x, 0) = \lambda(x)$ and the $j-$th order derivative $f^{(j)}_{0}(x, 0)$, $j=2, 3, \cdots$, of $f_0(x, s)$ at $s=0$ to be equal to $0$.
\end{itemize}

Thanks to
\[
	 \liminf_{|s| \rightarrow \infty}\frac{f_{0}(\cdot, s)}{s} =  \liminf_{|s| \rightarrow \infty}\left[\lambda(x) - e^{-\frac{1}{s^2}}\right] =  \lambda(x) -1> 0,
\]
we know $f_0(x, s)$ defined above satisfies the assumption (F1). Since we also have
\[
	f^{\prime}_{0}(x, s) = \lambda(x) - e^{-\frac{1}{s^2}} - \frac{2}{s^2}e^{-\frac{1}{s^2}},\quad 
	f^{\prime\prime}_{0}(x, s) = \frac{2}{s^3} e^{-\frac{1}{s^2}} - \frac{4}{s^5}e^{-\frac{1}{s^2}}
\]
for $s\neq0$, then we obtain
\[
	|f^{\prime\prime}_{0}(x, s)| < C\left( 1+ |s|\right), \quad  \forall s \in \mathbb{R},
\]
which means $f_0(x, s)$ satisfies the assumption (F2).
Furthermore, let $\hat{f}(x, s) = se^{-\frac{1}{s^2}}$. Then we have the following estimates, which means that $\hat{f}(x, s)$ satisfies the three conditions in Theorem~\ref{equilibria:infinity}. 
\begin{itemize}
\item[-] For $2<q< 6$, we have $|\hat{f}(s)|\leqslant C\left( 1+ |s|^{q-1} \right)$.
\item[-] Given any $R>0$, choosing $\kappa < 2+ \frac{2}{R^2}$ we have
\begin{equation*}
\begin{split}
	\left(s^{2}e^{-\frac{1}{s^2}} - \kappa\int^{s}_{0}\hat{f}(x, t) dt\right)^{\prime} &=2se^{-\frac{1}{s^2}}+\frac{2}{s}e^{-\frac{1}{s^2}}-\kappa s e^{-\frac{1}{s^2}} \\
	&= se^{-\frac{1}{s^2}}\left( 2+ \frac{2}{s^2} - \kappa \right)\\
	& > se^{-\frac{1}{s^2}}\left( \frac{2}{s^2} -  \frac{2}{R^2} \right) >0
\end{split}
\end{equation*}
for $|s| < R$. Then we have $s\hat{f}(x, s)-\kappa\hat{F}(s) \geqslant 0$.
\begin{equation*}
\begin{split}
	\lim_{|s|\rightarrow\infty}\frac{s^{2}e^{-\frac{1}{s^2}} - \kappa\int^{s}_{0}\hat{f}(x, t) dt}{s^2}& = \lim_{|s|\rightarrow\infty}\frac{s^{2}e^{-\frac{1}{s^2}} - \frac{ \kappa}{2}\Gamma\left(-1, \frac{1}{s^2}\right)}{s^2} \\
	&= \lim_{|s|\rightarrow\infty}\left( e^{-\frac{1}{s^2}}  - \frac{ \kappa}{2}\frac{\Gamma\left(-1, \frac{1}{s^2}\right)}{s^2}\right)\\
	&=1-\frac{\kappa}{2},
\end{split}
\end{equation*}
where $\Gamma(\cdot, \cdot)$ is the upper incomplete gamma function.
\item[-] By $-\hat{f}(x, s)= -se^{-s^2} =\hat{f}(x, -s)$, we know $\hat{f}(x, s)$ is odd.
\end{itemize}

\begin{remark}
The convergence results in~{\rm\cite[Corollary 1.4]{CEL2002}} and~{\rm\cite[Corollary 2.8]{chueshov2004}} cannot be applied to the equation~\eqref{wave_eq2} with this example of interior nonlinearity.
\end{remark}

\section{Non-convergence to an equilibrium}
\label{sec:nonconvergence}

In this section, we consider the form of initial-boundary-value problem for the wave equation~\eqref{wave_eq2} with $g(s)=s$ as follows
\begin{equation} \label{equ:nonconvergence}
\begin{split}
u_{tt} -\triangle u + f_0(x, u)&=0,\quad x\in\Omega,\quad t>0, \\
\partial_{\nu} u + u +u_t &= 0,\quad x\in\Gamma,\quad t>0,\\
u|_{t=0} = u_{0}(x), \quad u_{t}|_{t=0} &= u_{1}(x), \quad x\in \Omega,
\end{split}
\end{equation}
where the initial data $(u_0, u_1)^\top$ belongs to $\mathcal{D}$.
Define the $\omega$-limit set $\omega(u)$ of solution $u$ to equation~\eqref{equ:nonconvergence} by
\[
	\omega(u) = \{\varphi\in \mathcal{N}: \textrm{there exists a sequence $t_n\rightarrow \infty$ such that $u(t, \cdot) \rightarrow\varphi$ in $H^1(\Omega)$} \}.
\]
If $\omega(u)$ consists of just one equilibrium, that is,  $\omega(u) = \{\varphi\}$, then we say that the solution $u$ converges; otherwise, $u$ is said to be non-convergent.

To find what nonlinearities implies a non-convergent solution of equation~\eqref{equ:nonconvergence}, we need some preparations.

\begin{lemma} \label{eigenvalue:multiple}
Let $\Omega_1:=V_1\times V_2$, where $V_1=\{x\in\mathbb{R}^2: |x|<1\}$ and $V_2 = (0, \frac{5}{2}\pi)$. Assume that the domain $\Omega$ contains the cylinder $\Omega_1$.
Then there exists a smooth function $a(x)$ on $\bar{\Omega}$ such that the following statements hold.
\begin{itemize}
\item[{\rm(1)}] $0$ is an eigenvalue of operator $\mathcal{B}= \triangle_R + a(x)$ with the domain $D(\mathcal{B}) = D(\triangle_R)$, and it has a multiplicity of $2$.
\item[{\rm(2)}] Define $\Phi(x, \theta)$ by
\begin{equation}
\Phi(x, \theta)=\psi_1(x)\cos\theta + \psi_2(x)\sin\theta
\end{equation}
with $\psi_1$ and $\psi_2$ the eigenfunctions of operator $\mathcal{B}$ corresponding to the eigenvalue $0$. 
Then for each $\theta\in\mathbb{R}$ the function $\Phi(\cdot, \theta)$ assumes its positive maximum $M(\theta)$ over a subdomain $\bar{\Omega}_0$ of $\Omega$ at a unique point $x(\theta) \in \Omega_0$. Moreover, the Hessian matrix $\nabla^2\Phi(x(\theta), \theta)$ of $\Phi(\cdot, \theta)$ at $x(\theta)$ is negative definite, the map $\theta\rightarrow x(\theta)$ is injective on $[0, 2\pi)$ and $x^{\prime}(\theta)\neq 0$.
\end{itemize}
\end{lemma}
\begin{proof}
Let $\mu_2$ be the second eigenvalue of $-\triangle_R$ defined on $V_1$, and let $\upsilon_1 = 1$ be the principal eigenvalue of $-\triangle_R$ defined on $V_2$. Then $a_0 := \mu_2 + \upsilon_1$ is the second eigenvalue of $-\triangle_R$ defined on $\Omega_1$, which is of multiplicity $2$. This means that statement (1) holds for the case $\Omega =\Omega_1$. 

To verify statement (2) in the case $\Omega =\Omega_1$, we compute the corresponding orthonormal basis of $\mathrm{Ker}\mathcal{B}$ as follows
\begin{equation} \label{explicit_eigenfunctions}
	\psi_i(x_1, x_2, x_3) = cJ(r)\frac{x_i}{r}H(x_3), \quad i=1,2,
\end{equation}
where $c$ is a suitable normalizing constant, $r= \sqrt{x_1^2 + x_2^2}\in(0, 1)$, $J(r)\frac{x_i}{r}$ is the eigenfunction of $-\triangle_R$ defined on $V_1$  corresponding to $\mu_2$, and $H(x_3)$ is the eigenfunction of $-\triangle_R$ defined on $V_2$ corresponding to $\upsilon_1$. In particular, $H(\cdot)$ is the solution of the following equation
\begin{equation}\label{eigenfunceion_3dim}
\left.
\begin{array}{c}
-H^{\prime\prime}(x_3) =   H(x_3)\\
-H^{\prime}(0) + H(0) = 0\\
H^{\prime}(\frac{5}{2}\pi) + H(\frac{5}{2}\pi) = 0
\end{array}
\right\}
\Longrightarrow
H(x_3) = \cos x_3- \sin x_3,
\end{equation}
and $J(r) = J_1(\mathrm{y}_{1}r)$. Here, $J_1(\cdot)$ is the Bessel function of the first kind, and $\mathrm{y}_1$ is the first positive root of $J_0(\cdot)$, which is derived from the following boundary condition at $r=1$ and iteration formula of the Bessel function
\[
\left.
\begin{array}{c}
\mathrm{y}J^{\prime}_1(\mathrm{y}) + J_1(\mathrm{y}) = 0\\
J^{\prime}_1(\mathrm{y}) = J_0(\mathrm{y}) - \frac{J_1(\mathrm{y})}{\mathrm{y}}
\end{array}
\right\}
\Longrightarrow
J_0(\mathrm{y}) = 0.
\]
By using the polar coordinates $x_1 = r\cos\vartheta$ and $x_2 = r\sin\vartheta$, then $\Phi(x, \theta)$ can be rewritten as
\begin{equation*}
\begin{split}
\Phi(x, \theta) &= cJ(r)\left( \frac{x_1}{r} \cos\theta + \frac{x_2}{r}\sin\theta \right)  H(x_3)\\
&=cJ(r)\left( \cos\vartheta \cos\theta + \sin\vartheta\sin\theta \right) H(x_3)\\
&= cJ(r)\cos(\vartheta - \theta) H(x_3).
\end{split}
\end{equation*}
Due to the properties of the Bessel function, we know there exists a unique critical point $r^{\ast}\in(0, 1)$ such that $J_1^{\prime}(\mathrm{y}_{1}r^{\ast}) = 0$ and $J_1^{\prime\prime}(\mathrm{y}_{1}r^{\ast}) < 0$. Moreover, from~\eqref{eigenfunceion_3dim} we also know $H(\cdot)$ has a unique critical point $x^{\ast}_3 = \frac{7}{4}\pi \in(0, \frac{5}{2}\pi)$ satisfying $H^{\prime}(x^{\ast}_3) = 0$ and $H^{\prime\prime}(x^{\ast}_3)=-\sqrt{2} < 0$.
Consequently, $\Phi(\cdot, \theta)$ has a local maximum point
\[
	x(\theta) = (r^{\ast}\cos\theta, r^{\ast}\sin\theta,  \frac{7}{4}\pi).
\]
Let $\Omega_0$ be any subdomain of $\Omega_1$, which contains the circle $\{x(\theta): \theta\in[0, 2\pi] \}$. Then statement (2) holds true for such $\Omega_0\subset \Omega_1$.

Now, we proceed to extend the previous result to the case of a general domain $\Omega\supsetneqq\Omega_1$. To do that, let us introduce a smooth function $b(x)$ satisfying $b=0$ on $\Omega_0$ and $b<0$ on $\mathbb{R}^3\backslash\bar{\Omega}_1$, and a sequence $\gamma_k\in\mathbb{R}$ with $\gamma_k\rightarrow\infty$ as $k\rightarrow\infty$. From Theorem 2.5 and Theorem 3.4 in \cite{prizzi1998inverse} it implies that there exists $k_0$ and a sequence functions $\iota_k(x)\in C^\infty(\mathbb{R}^3)$, $k = k_0, k_0 +1, \cdots$, such that $\|\iota_k(x)\|_{L^\infty(\Omega)}\rightarrow 0$ as $k\rightarrow\infty$, and the following properties hold.
\begin{itemize}
\item[(I)] If $\mathcal{B}_k = \triangle_{R} + a_k(x)$ with $a_k(x) = a_0 + \gamma_k b(x) + \iota_k(x)$, then the second eigenvalue of $\mathcal{B}_k$ is zero, which is of multiplicity $2$.
\item[(II)] There is an $L^2(\Omega)$-orthonormal basis $\psi^{k}_i$, $i=1, 2$, of $\mathrm{Ker}\mathcal{B}_k$ such that $\psi^{k}_i\rightarrow\tilde{\psi}_i$, $i=1, 2$, in $H^1(\Omega)$ as $k\rightarrow\infty$, where $\tilde{\psi}_i$, $i=1, 2$, are trivial extensions of $\psi_i$ given in~\eqref{explicit_eigenfunctions}. 
\end{itemize}
Based on this claim, we know statement (1) is satisfied if $\mathcal{B}=\mathcal{B}_k$. Set 
\[
\Phi^k(x, \theta) = \psi^k_1(x)\cos\theta + \psi^k_2(x)\sin\theta.
\]
From property (II) and the standard elliptic interior estimates it follows that the following convergence holds in $C^2(\bar{\Omega}_0)$ 
\[
	\Phi^k(x, \theta)\rightarrow \Phi(x, \theta), \quad k\rightarrow\infty,
\]
which is also uniform with respect to $\theta\in[0, 2\pi]$. Thus, for each $\theta_0\in\mathbb{R}$ there exists a neighborhood $W$ of $\theta_0$ and $k_0 = k_0(\theta_0)$ such that one has the following results: (i) for $k>k_0$ and $\theta\in W$, $\Phi^k(\cdot, \theta)$ possesses its positive maximum over $\bar{\Omega}_0$ at a unique point $x^{k}(\theta)\in\Omega_0$, and its Hessian matrix $\nabla^2\Phi^k(x^k(\theta), \theta)$ is negative definite; (ii) $x^k(\theta)\rightarrow x(\theta)$ and $\frac{\mathrm{d}}{\mathrm{d}\theta}x^k(\theta)\rightarrow \frac{\mathrm{d}}{\mathrm{d}\theta}x(\theta)$ as $k$ goes to infinity uniformly for $\theta\in W$. 

By the periodicity of maps $\theta\mapsto\Phi^k(\cdot, \theta)$, $\theta\mapsto x^k(\theta)$, and the compactness argument, we know $k_0$ can be chosen independent of $\theta_0$.Therefore, if we choose $k$ to be sufficiently large, then statement~(2) holds for $\Phi(x, \theta)$ and $x(\theta)$ replaced by $\Phi^k(x, \theta)$ and $x^k(\theta)$.
\end{proof}

\begin{remark}
The map $\theta\rightarrow M(\theta)$ is a positive $2\pi-$periodic function of class $C^{\infty}$.
\end{remark}

Throughout this section, we assume that the domain $\Omega$ satisfies the condition in Lemma~\ref{eigenvalue:multiple}. Similarly to equation~\eqref{evolution_eq}, we can also reformulate the wave equation~\eqref{equ:nonconvergence} as the following abstract form in $\mathcal{D}$
\begin{equation} \label{evolu_equ:nonconvergence}
\partial_t\left(
\begin{array}{c}
u \\
u_t
\end{array}
\right)
+\tilde{\mathcal{A}}
\left(
\begin{array}{c}
u \\
u_t
\end{array}
\right)
+
\mathcal{F}\left(
\begin{array}{c}
u \\
u_t
\end{array}
\right)
=0
\end{equation}
with
 \begin{equation*}
\tilde{\mathcal{A}}\left(
\begin{array}{c}
u \\
u_t
\end{array}
\right)
=
\left(
\begin{array}{c}
-u_t\\
-\mathcal{B}u + T^{\ast}\left(Tu_t\right)
\end{array}
\right),
\quad
\mathcal{F}\left(
\begin{array}{c}
u \\
u_t
\end{array}
\right)
=\left(
\begin{array}{c}
0 \\
\tilde{f}_a(u)
\end{array}
\right),
\end{equation*}
and $\tilde{f}_a: u\mapsto f_a(x, u) := a(x)u + f_0(x, u)$. Note that $\tilde{\mathcal{A}}$ is a self-adjoint, and there exist constants $C>0$ and $\lambda>\max\{|a(x)|: x\in\bar{\Omega}\}$ such that the following estimate holds
\[
\left\|(\lambda\mathcal{I} + \tilde{\mathcal{A}})\left(
\begin{array}{c}
u \\
v
\end{array}
\right)\right\|_{\mathcal{H}}\geqslant C\|(u, v)^\top\|_{\mathcal{H}}
\]
for $(u, v)^\top\in\mathcal{D}$.
Consequently, the resolvent set of $\tilde{\mathcal{A}}$ is nonempty. In terms of the fact that $H^2(\Omega)\times H^1(\Omega)$ is compact embedded into $\mathcal{H}$, the operater $\tilde{\mathcal{A}}$ has a compact resolvent. It follows from the Fredholm alternative that the zero point is an isolated point in the spectrum of $\tilde{\mathcal{A}}$, and $\mathrm{Ker}\tilde{\mathcal{A}}$ is finite dimensional.  
By Lemma~\ref{eigenvalue:multiple}, we can introduce an $L^2(\Omega)$-orthonormal basis of $\mathrm{Ker}\mathcal{B}\subset \mathcal{D}$, which are still denoted by $\psi_1$ and $\psi_2$. Due to
 \begin{equation*}
 0=
\tilde{\mathcal{A}}\left(
\begin{array}{c}
u \\
v
\end{array}
\right)
\Longrightarrow
\mathcal{B}u=0, \quad v=0,
\end{equation*}
then the eigenspace $E_1$ of $\tilde{\mathcal{A}}$ associated to the eigenvalue $0$ is given by
\[
E_1= \mathrm{Ker}\tilde{\mathcal{A}} = 
\left\{
\left(
\begin{array}{c}
\xi\cdot\psi \\
0
\end{array}
\right):
\xi = (\xi_1, \xi_2)\in\mathbb{R}^2, \psi = (\psi_1, \psi_2)
\right\}.
\]
We denote by $\mathcal{P}: \mathcal{H}\rightarrow E_1$ the spectral projection of $\tilde{\mathcal{A}}$ associated with the spectral set $\{0\}$, and $E_2$ the ranges of $\mathcal{I}-\mathcal{P}$. 
From the growth assumption (A1) on $f_0$ one can easily see that the Nemytski mapping $s\mapsto f_0(x, s)$ is locally Lipschitz from $H^2(\Omega)$ to $H^1(\Omega)$. Furthermore, the assumption that $f_0\in C^m_{\mathrm{b}}(\bar{\Omega}\times\mathbb{R})$ implies from $H^2(\Omega)\hookrightarrow  C(\bar{\Omega})$ that $f_a$ is also of class $C^m_{\mathrm{b}}(\bar{\Omega}\times\mathbb{R})$. Finally, we have the following equality
\[
	\|\mathcal{F}\|_{C^{m-1}_{\mathrm{b}}(\mathcal{D}\rightarrow \mathcal{D})}\leqslant \|\tilde{f}_a\|_{C^{m-1}_{\mathrm{b}}(H^2\rightarrow H^1)} \leqslant \|f_a\|_{C^m_{\mathrm{b}}(\bar{\Omega}\times \mathbb{R})}. 
\]

\begin{lemma}[The centre-manifold reduction]\label{cm_reduction}
If $f_0$ belongs to $C^m_{\mathrm{b}}(\bar{\Omega}\times\mathbb{R})$ and $\|\mathcal{F}\|_{C^{1}_{\mathrm{b}}(\mathcal{D}\rightarrow \mathcal{D})} < c_m$, then there exists a $C^m_{\mathrm{b}}$ map 
\[
	\Lambda: \mathbb{R}^2\rightarrow \mathcal{D},\quad  \xi\mapsto \Lambda(\xi)=(\Lambda_1(\xi), \Lambda_2(\xi))^{\top}
\]
with image contained in $\mathcal{D}\cap E_2$ such that the following properties hold.
\begin{itemize}
\item[{\rm(1)}] The equation~\eqref{evolu_equ:nonconvergence} possesses a unique global center manifold
\[
\mathcal{M}_{f_0} = \left\{
\left(
\begin{array}{c}
\xi\cdot\psi \\
0
\end{array}
\right)
+\Lambda(\xi): \xi\in\mathbb{R}^2 
\right\}.
\]
\item[{\rm(2)}] If $U_0 \in \mathcal{M}_{f_0}$ for some $\xi_0\in\mathbb{R}^2$, then the solution of equation \eqref{evolu_equ:nonconvergence} subject to the initial condition $U(0) = U_0$ is given by
\[
U(t):=U(\xi(t))=
\left(
\begin{array}{c}
\xi(t)\cdot\psi \\
0
\end{array}
\right)
+\Lambda(\xi(t)),
\]
where $\xi(t)$ is the solution of the following ODE
\begin{equation} \label{cmr_ode}
\dot{\xi}(t) = h(\xi(t)), \quad \xi(0) = \xi_0
\end{equation}
with the $C^1$ function $h: \mathbb{R}^2\rightarrow \mathbb{R}^2$ given by
\begin{equation} \label{cmr_ode1}
\left(
\begin{array}{c}
h(\xi)\cdot\psi \\
0
\end{array}
\right)
=
-\mathcal{P}\mathcal{F}(U).
\end{equation}
\item[{\rm(3)}] For each $\xi\in\mathbb{R}^2$, one has
\begin{equation}\label{cmr_ode2}
\begin{split}
	\nabla_{\xi} U \cdot h(\xi) + \tilde{\mathcal{A}}U(t) + \mathcal{F}(U(t)) &= 0,\\
	\nabla_{\xi} \Lambda \cdot h(\xi) + \tilde{\mathcal{A}}\Lambda(\xi) + (\mathcal{I}-\mathcal{P})\mathcal{F}(U(\xi)) &= 0.
\end{split}
\end{equation}
\item[{\rm(4)}] There is a positive constant $C$ independent of $f$ such that
\begin{equation}\label{cmr_estimate}
	\|\Lambda\|_{C^{j-1}_{\mathrm{b}}(\mathbb{R}^2\rightarrow \mathcal{D})} \leqslant C\|\mathcal{F}\|_{C^{1}_{\mathrm{b}}(\mathcal{D}\rightarrow \mathcal{D})}
\end{equation}
 for $j=1, 2, \cdots, m-1$.
\end{itemize}
\end{lemma}
\begin{proof}
Since $0$ is an isolated point in the spectrum of $\tilde{\mathcal{A}}$ and $\mathrm{Ker}\tilde{\mathcal{A}}$ is finite dimensional, then there exists a positive constant $c > 0$ such that the spectrum of $\tilde{\mathcal{A}}|_{E_2}$, the restriction of $\tilde{\mathcal{A}}$ on $E_2$, belongs to $(-\infty, -c]$, which implies $\tilde{\mathcal{A}}|_{E_2}$ is the infinitesimal generator of a strongly continuous semigroup which is exponential decay. By the result in~\cite[Section 3]{Vanderbauwhede1992}, we obtain the properties (1) and (2) of the lemma.

Notice that $\dot{U}(t) = \nabla_{\xi} U \cdot h(\xi)$. Then we get the first equality of~\eqref{cmr_ode2}. Since we also have
\begin{equation*}
\begin{aligned}
0=\dot{U}(t) +\tilde{\mathcal{A}}U(t) + \mathcal{F}(U(t)) &= 
\left(
\begin{array}{c}
h(\xi)\cdot\psi \\
0
\end{array}
\right)
+ \dot{\Lambda}(\xi) + \tilde{\mathcal{A}}\Lambda(\xi) + \mathcal{F}(U)\\
&=-\mathcal{P}\mathcal{F}(U) + \dot{\Lambda}(\xi) + \tilde{\mathcal{A}}\Lambda(\xi) + \mathcal{F}(U) \\
&=\dot{\Lambda}(\xi)+ \tilde{\mathcal{A}}\Lambda(\xi) + (\mathcal{I}-\mathcal{P})\mathcal{F}(U),
\end{aligned}
\end{equation*}
and $\dot{\Lambda}(\xi) = \nabla_{\xi} \Lambda \cdot h(\xi)$, the second equality of~\eqref{cmr_ode2} is obtained.
The smoothness of the center manifold (cf. \cite[Theorem 2]{Vanderbauwhede1992}) guarantees property (4).
\end{proof}

Let us introduce the energy functional of \eqref{evolu_equ:nonconvergence} as
\[
\mathcal{G}(u, v)=\frac{1}{2}\|(u, v)^{\top}\|^2_{\mathcal{H}} + \int_{\Omega}F_0(x, u)\mathrm{d}x.
\]
If $f_0$ satisfies the assumption of Lemma~\ref{cm_reduction}, then the functional $\Xi(\xi)$ defined by
\begin{equation*} \label{Lyapunov_function}
\xi\mapsto\Xi(\xi) :=\mathcal{G}(\xi\cdot\psi + \Lambda_1(\xi), \Lambda_2(\xi))\\
\end{equation*}
is a Lyapunov functional of \eqref{cmr_ode} due to the decreasing property of $\mathcal{G}$. 

In the following, we are about to show that $\Xi(\xi)$ satisfies some geometric conditions, which implies~\eqref{cmr_ode} has a trajectory with its $\omega$-limit set equal to a unit circle $S^1$. To do so, we use the following polar coordinates on $\mathbb{R}^2$
\[
	\xi_1 = \rho\cos\theta, \quad \xi_2 = \rho\sin\theta,
\]
and introduce the notations for a disk and a circle as follows
\begin{equation*}
\begin{split}
	S &= \{(\rho\cos\theta, \rho\sin\theta): \rho M(\theta) =1, \theta\in[0, 2\pi)\},\\
	B_{\delta} &= \{(\rho\cos\theta, \rho\sin\theta): \rho M(\theta) < 1+\delta, \theta\in[0, 2\pi)\}.
\end{split}
\end{equation*}
\begin{proposition}\label{ode_nonconvergence}
Assume $\Xi(\xi)$ satisfies the following conditions.
\begin{itemize}
\item[{\rm{(1)}}] $\Xi$ vanishes on $\bar{B}_0$.
\item[{\rm{(2)}}] There are $C^1$ functions $l_1(\theta)$, $l_2(\theta)$ on $(0, \infty)$ such that 
\begin{equation*}
\begin{split}
	l_2(\theta) > l_1(\theta) > l_2(\theta + 2\pi) > \frac{1}{M(\theta)}, \quad &\theta\in(0, \infty),\\
	l_2(\theta) - \frac{1}{M(\theta)}\rightarrow 0, \quad &\theta\rightarrow\infty,
\end{split}
\end{equation*}
and for some $\delta > 0$, the level set $\{\xi: \Xi(\xi) = 0\}$ in $\mathrm{B}_{\delta}$ is identical to the union
\[
	\bar{B}_0\cap\left(\mathrm{O}_1\cap B_{\delta} \right)\cap\left(\mathrm{O}_2\cap B_{\delta} \right)
\]
with $\mathrm{O}_i := \{(\rho\cos\theta, \rho\sin\theta): \rho  = l_i(\theta), \theta\in (0, \infty)\}$, $i=1,2.$
\item[{\rm{(3)}}] $\nabla_{\xi}\Xi \neq 0$ on $B_{\delta}\backslash B_{0}$, $\delta$ given in the condition {\rm{(2)}}.
\end{itemize}
If the set of equilibria of~\eqref{cmr_ode} coincides with $\bar{B}_0$. Then there exists a point $\xi_0\in B_{\delta}$ such that equation~\eqref{cmr_ode} with $\xi(0)=\xi_0$ has a bounded solution $\xi(t)$ on $[0, \infty)$, and its $\omega$-limit set equals $S$.
\end{proposition}

In fact, this proposition describes a class of planar ordinary differential equations with non-convergent trajectories, which is independent of the form of the evolution equation. For the detailed proof of this proposition, we refer to \cite[Proposition 2.1]{polavcik2002nonconvergent}. Here, we use the modified version proposed in \cite[Secion 2.1]{jendoubi2003non}.

The following two lemmas provides an analysis of the Lyapunov functional $\Xi(\xi)$.
\begin{lemma}\label{decomposition_1}
Any equilibrium of~\eqref{cmr_ode} is a critical point of the functional $\Xi$. Moreover, if $f_0$ is of class $ C^3_{\mathrm{b}}(\bar{\Omega}\times\mathbb{R})$ and $\|f_a\|_{C^2_{\mathrm{b}}(\bar{\Omega}\times\mathbb{R})}< c^{\ast} < c_3$ with $c^{\ast}$ sufficiently small, then one has
\[
	\Xi(\xi) = \Upsilon_1(\xi) +  \int_{\Omega}F_a(x, \xi\cdot\psi)\mathrm{d}x, \quad \xi\in\mathbb{R}^2,
\]
where $\Upsilon_1(\xi)$ is a $C^2$ function on $\mathbb{R}^2$ satisfying
\begin{equation} \label{estimates_lyapunov_3}
|\Upsilon_1(\xi)| + |\partial_{\xi_i}\Upsilon_1(\xi)|\leqslant C\max\{1, \|f_a\|^2_{C^3_{\mathrm{b}}(\bar{\Omega}\times\mathbb{R})}\}\left( \sup_{x\in\Omega}|f_a(x, \xi\cdot\psi)|^2 + \sup_{x\in\Omega}|f^{\prime}_a(x, \xi\cdot\psi)| \right)
\end{equation}
with $i=1, 2$ and $C$ a positive constant.
\end{lemma}
\begin{proof}
To facilitate reading, we divide the proof into the following three steps.

\textbf{Step 1}. Combining integration by parts and the fact that $\xi\cdot\psi\in\mathrm{Ker}B$ gives
\begin{equation} \label{lyapunov_1}
\begin{split}
\Xi(\xi) &= \frac{1}{2}\left\|\left(\xi\cdot\psi + \Lambda_1(\xi), \Lambda_2(\xi)\right)^{\top}\right\|^2_{\mathcal{H}} + \int_{\Omega}F_0(x, \xi\cdot\psi + \Lambda_1(\xi))\mathrm{d}x\\
&= \frac{1}{2}\left(\int_{\Omega}|\nabla (\xi\cdot\psi + \Lambda_1(\xi))|^2\mathrm{d}x + \int_{\Gamma}|\xi\cdot\psi +\Lambda_1(\xi)|^2\mathrm{d}S + \int_{\Omega}|\Lambda_2(\xi)|^2\mathrm{d}x\right) \\
&\quad + \int_{\Omega}F_0(x, \xi\cdot\psi + \Lambda_1(\xi))\mathrm{d}x\\
&= \frac{1}{2}\left(\int_{\Omega}|\nabla  \Lambda_1(\xi)|^2\mathrm{d}x - a(x)\int_{\Omega}|  \Lambda_1(\xi)|^2\mathrm{d}x + \int_{\Gamma}|\Lambda_1(\xi)|^2\mathrm{d}S + \int_{\Omega}|\Lambda_2(\xi)|^2\mathrm{d}x\right) \\
&\quad+\int_{\Omega}F_a(x, \xi\cdot\psi + \Lambda_1(\xi))\mathrm{d}x
\end{split}
\end{equation}
with $F_a(x, u)=\int_{0}^{u}f_a(x, s)\mathrm{d}s$.
Differentiating~\eqref{lyapunov_1} with respect to $\xi_i$, $i=1, 2$, we obtain
\begin{equation}\label{lyapunov_2}
\begin{split}
\partial_{\xi_i}\Xi(\xi) &=  -\int_{\Omega}\mathcal{B}\Lambda_1(\xi)\partial_{\xi_i}\Lambda_1(\xi)\mathrm{d}x + \int_{\Omega}\Lambda_2(\xi)\partial_{\xi_i}\Lambda_2(\xi)\mathrm{d}x-\int_{\Gamma} \Lambda_2(\xi)\partial_{\xi_i}\Lambda_1(\xi)\mathrm{d}S\\
&\quad + \int_{\Omega}f_a(x, \xi\cdot\psi + \Lambda_1(\xi))(\psi_i + \partial_{\xi_i}\Lambda_1(\xi))\mathrm{d}x\\
&=  -\int_{\Omega}[\mathcal{B}\Lambda_1(\xi)-f_a(x, \xi\cdot\psi + \Lambda_1(\xi))](\psi_i + \partial_{\xi_i}\Lambda_1(\xi))\mathrm{d}x - \int_{\Gamma}\psi_i \Lambda_2(\xi)\mathrm{d}S\\
&\quad + \int_{\Omega}\Lambda_2(\xi)\partial_{\xi_i}\Lambda_2(\xi)\mathrm{d}x-\int_{\Gamma} \Lambda_2(\xi)\partial_{\xi_i}\Lambda_1(\xi)\mathrm{d}S.
\end{split}
\end{equation}
The first equality of~\eqref{cmr_ode2} can be written in component wise
\begin{equation}\label{cmr_ode3}
\begin{split}
	  h(\xi)\cdot \psi + \nabla_{\xi}\Lambda_1 \cdot h(\xi) &= \Lambda_2(\xi),\\
	\nabla_{\xi}\Lambda_2 \cdot h(\xi) &= \mathcal{B}\Lambda_1 - T^\ast(T\Lambda_2) - f_a(x, h(\xi)\cdot \psi +\Lambda_1).
\end{split}
\end{equation}
For every $\xi_0$, $h(\xi_0) = 0$ implies from~\eqref{cmr_ode3} that 
\[
\Lambda_2(\xi_0) = 0, \quad \mathcal{B}\Lambda_1(\xi_0) - f_a(x, \Lambda_1(\xi_0)) = 0.
\]
Consequently, it deduce from~\eqref{lyapunov_2} that $\nabla_{\xi}\Xi(\xi_0) = 0$. This completes the first part of the lemma.

\textbf{Step 2}. Due to the following equality
\begin{equation*}
\begin{split}
 \int_{\Omega}F_a(x, \xi\cdot\psi + \Lambda_1(\xi))\mathrm{d}x -\int_{\Omega}F_a(x, \xi\cdot\psi)\mathrm{d}x &=  \int_{\Omega}\int_{\xi\cdot\psi}^{\xi\cdot\psi + \Lambda_1(\xi)}f_a(x, y)\mathrm{d}y\mathrm{d}x\\
 &=\int_{\Omega}\Lambda_1(\xi))\int^{1}_{0}f_a(x, \xi\cdot\psi + s\Lambda_1(\xi))\mathrm{d}s\mathrm{d}x,	
\end{split}
\end{equation*}
then from~\eqref{lyapunov_1} the Lyapunov functional $\Xi(\xi)$ can be reformulated as
\begin{equation*}
\Xi(\xi) =  \Upsilon_1(\xi) +  \int_{\Omega}F_a(x, \xi\cdot\psi)\mathrm{d}x
\end{equation*}
with
\begin{equation}\label{lyapunov_3}
\begin{split}
\Upsilon_1(\xi) &:= \frac{1}{2}\left(\int_{\Omega}|\nabla  \Lambda_1(\xi)|^2\mathrm{d}x - a(x)\int_{\Omega}|  \Lambda_1(\xi)|^2\mathrm{d}x + \int_{\Gamma}|\Lambda_1(\xi)|^2\mathrm{d}S + \int_{\Omega}|\Lambda_2(\xi)|^2\mathrm{d}x\right) \\
&\qquad + \int_{\Omega}\Lambda_1(\xi))\int^{1}_{0}f_a(x, \xi\cdot\psi + s\Lambda_1(\xi))\mathrm{d}s\mathrm{d}x. 
\end{split}
\end{equation}
From differentiating~\eqref{lyapunov_3} with respect to $\xi_i$ it deduces that
\begin{equation}\label{diff_lyapunov_3}
\begin{split}
\partial_{\xi_i}\Upsilon_1(\xi) &=  -\int_{\Omega}\mathcal{B}\Lambda_1(\xi)\partial_{\xi_i}\Lambda_1(\xi)\mathrm{d}x + \int_{\Omega}\Lambda_2(\xi)\partial_{\xi_i}\Lambda_2(\xi)\mathrm{d}x-\int_{\Gamma} \Lambda_2(\xi)\partial_{\xi_i}\Lambda_1(\xi)\mathrm{d}S\\
&\quad + \int_{\Omega}\Lambda_1(\xi))\int^{1}_{0} f^{\prime}_a(x, \xi\cdot\psi + s\Lambda_1(\xi))(\psi_i + s\partial_{\xi_i}\Lambda_1(\xi))\mathrm{d}s\mathrm{d}x\\
&\quad + \int_{\Omega}\partial_{\xi_i}\Lambda_1(\xi))\int^{1}_{0}f_a(x, \xi\cdot\psi + s\Lambda_1(\xi))\mathrm{d}s\mathrm{d}x.
\end{split}
\end{equation}

To estimate $\Upsilon_1(\xi)$ and $\partial_{\xi_i}\Upsilon_1(\xi)$, we need the following estimates
\begin{equation}\label{estimates_lyp_3_1}
\begin{split}
	\| \Lambda_1(\xi)\|_{H^2} + \| \Lambda_2(\xi)\|&\leqslant C_1\|\tilde{f}_a(\xi\cdot\psi)\|,\\
	\|\tilde{f}_a(x, \xi\cdot\psi + s\Lambda_1(\xi))\|&\leqslant C_2 \|\tilde{f}_a(\xi\cdot\psi)\|,\\
	\|\partial_{\xi_i} \Lambda_1(\xi)\|_{H^2} + \| \partial_{\xi_i}\Lambda_2(\xi)\|&\leqslant C_3\left(\|f_a\|_{C^3_{\mathrm{b}}(\bar{\Omega}\times\mathbb{R})}\|\tilde{f}_a(\xi\cdot\psi)\| + \|\tilde{f}^{\prime}_a(\xi\cdot\psi)\psi_i \| \right), \\
	\|\partial_{\xi_i}\tilde{f}_a(x, \xi\cdot\psi + s\Lambda_1(\xi))\|&\leqslant C_4\left(\|f_a\|_{C^3_{\mathrm{b}}(\bar{\Omega}\times\mathbb{R})}\|\tilde{f}_a(\xi\cdot\psi)\| + \|\tilde{f}^{\prime}_a(\xi\cdot\psi)\psi_i \| \right).
\end{split}
\end{equation}
By~\eqref{lyapunov_3} and the first two estimates of~\eqref{estimates_lyp_3_1} we get
\[
|\Upsilon_1(\xi)| \leqslant C_5\|\tilde{f}_a(\xi\cdot\psi)\|^2.
\]
Similarly, by~\eqref{diff_lyapunov_3} and the last two estimates of~\eqref{estimates_lyp_3_1} we have
\[
|\partial_{\xi_i}\Upsilon_1(\xi)| \leqslant C_6\left(\|f_a\|^2_{C^3_{\mathrm{b}}(\bar{\Omega}\times\mathbb{R})}\|\tilde{f}_a(\xi\cdot\psi)\|^2 + \|\tilde{f}^{\prime}_a(\xi\cdot\psi)\psi_i \|^2 \right).
\]
We also note that
\begin{equation*}
	\|\tilde{f}_a(\xi\cdot\psi)\|\leqslant C_7\sup_{x\in\Omega}|f_a(x, \xi\cdot\psi)|,  \quad \|\tilde{f}^{\prime}_a(\xi\cdot\psi)\|\leqslant C_8 \sup_{x\in\Omega}|f^{\prime}_a(x, \xi\cdot\psi)|.
\end{equation*}
Thus, the expected inequality~\eqref{estimates_lyapunov_3} can be obtained.

\textbf{Step 3}. The remaining part of this proof is to show the estimates~\eqref{estimates_lyp_3_1} hold true. 

Let $\tilde{\mathcal{B}}$ be the restriction of the operator $\mathcal{B}$ to $(\mathrm{Ker}\mathcal{B})^{\bot}$. Then its inverse operator $\tilde{\mathcal{B}}^{-1}: (\mathrm{Ker}\mathcal{B})^{\bot}\rightarrow H^2(\Omega)$ is bounded. From~\eqref{cmr_ode3} we have
\begin{equation} \label{eq:lambdas}
\begin{split}
	\Lambda_1(\xi) &= \tilde{\mathcal{B}}^{-1}\left(\nabla_{\xi}\Lambda_2 \cdot h(\xi) + T^\ast(T\Lambda_2) + f_a(x, h(\xi)\cdot \psi + \Lambda_1)\right),\\
	\Lambda_2(\xi) &= h(\xi)\cdot \psi + \nabla_{\xi}\Lambda_1 \cdot h(\xi).
\end{split}
\end{equation}

By~\eqref{cmr_ode1} and ~\eqref{cmr_estimate}, the second equality of~\eqref{eq:lambdas} yields
\begin{equation} \label{eq:lambdas_2}
\begin{split}
\|\Lambda_2(\xi)\|&\leqslant \|h(\xi)\cdot \psi \| + |\nabla_{\xi}\Lambda_1|\|h(\xi)\|\\
&\leqslant C_8\|\tilde{f}_a(\xi\cdot\psi + \Lambda_1(\xi))\|.
\end{split}
\end{equation}
A similar estimate in the first equality of~\eqref{eq:lambdas} gives
\begin{equation} \label{eq:lambdas_1}
\begin{split}
\|\Lambda_1(\xi)\|_{H^2} &\leqslant C_9\left(\|\tilde{f}_a(\xi\cdot\psi + \Lambda_1(\xi))\| + \|\Lambda_2(\xi)\|\right)\\
	&\leqslant C_{10}\|\tilde{f}_a(\xi\cdot\psi + \Lambda_1(\xi))\|.
\end{split}
\end{equation}
If $\|f_a\|_{C^2_{\mathrm{b}}(\bar{\Omega}\times\mathbb{R})}< c^{\ast}$ with $c^\ast$ small enough, then by the mean-value theorem we have
\begin{equation} \label{estimate:f_a}
C_{10}\|\tilde{f}_a(\xi\cdot\psi + \Lambda_1(\xi))\|\leqslant C_{10}\|\tilde{f}_a(\xi\cdot\psi)\| +\frac{1}{2}\|\Lambda_1(\xi)\|.
\end{equation}
Thus, from~\eqref{eq:lambdas_1}, \eqref{eq:lambdas_2} and~\eqref{estimate:f_a} it implies the first two estimates of~\eqref{estimates_lyp_3_1}.

By differentiating \eqref{eq:lambdas} with respect to $\xi_i$, it implies from~\eqref{cmr_ode1} and~\eqref{cmr_estimate} that
\begin{equation} \label{estimate:partial_xi}
	\|\partial_{\xi_i} \Lambda_2(\xi)\|\leqslant C_{11}\mathcal{Q}, \quad \|\partial_{\xi_i} \Lambda_1(\xi)\|_{H^2} \leqslant C_{12}\left( \mathcal{Q} + \|\partial_{\xi_i} \Lambda_2(\xi)\|\right)
\end{equation}
with
\begin{equation*}
	\mathcal{Q}:= \|\tilde{f}^{\prime}_a(\xi\cdot\psi + \Lambda_1(\xi))(\psi_i + \partial_{\xi_i}\Lambda_1(\xi))\| +\|f_a\|_{C^3_{\mathrm{b}}(\bar{\Omega}\times\mathbb{R})}\|\tilde{f}_a(\xi\cdot\psi + \Lambda_1(\xi))\|.
\end{equation*}
Combining the first two estimates of~\eqref{estimates_lyp_3_1} and~\eqref{estimate:partial_xi} with the mean-value theorem gives 
\begin{equation*}
\begin{split}
	&\|\partial_{\xi_i} \Lambda_1(\xi)\|_{H^2} + \|\partial_{\xi_i} \Lambda_2(\xi)\| \\
	&\leqslant C_{13}\mathcal{Q} \\
	&\leqslant C_{14}\left(\|\tilde{f}^{\prime}_a(\xi\cdot\psi)\psi_i \| + \|f_a\|_{C^3_{\mathrm{b}}(\bar{\Omega}\times\mathbb{R})}\|\Lambda_1(\xi)\| + \|f_a\|_{C^3_{\mathrm{b}}(\bar{\Omega}\times\mathbb{R})}\|\tilde{f}_a(\xi\cdot\psi)\| + c^\ast\|\partial_{\xi_i} \Lambda_1(\xi)\|_{H^2} \right)\\
	&\leqslant C_{15}\left(\|\tilde{f}^{\prime}_a(\xi\cdot\psi)\psi_i \| + \|f_a\|_{C^3_{\mathrm{b}}(\bar{\Omega}\times\mathbb{R})}\|\tilde{f}_a(\xi\cdot\psi)\| + c^\ast\|\partial_{\xi_i} \Lambda_1(\xi)\|_{H^2} \right),
\end{split}
\end{equation*}
which implies the third estimate of~\eqref{estimates_lyp_3_1} by choosing $c^\ast$ small enough. The last estimate of~\eqref{estimates_lyp_3_1} follows from the third estimate and the mean-value theorem. This concludes the proof.
\end{proof}

Next, we are to give a computation of the integral $\int_{\Omega}F_a(x, \xi\cdot\psi)\mathrm{d}x $.

\begin{lemma} \label{decomposition_2}
Given any $c>0$, there exists a function $f_a$ belonging to $C^{\infty}(\bar{\Omega}\times\mathbb{R})\cap {C^3_{\mathrm{b}}(\bar{\Omega}\times\mathbb{R})}$ such that $\|f_a\|^2_{C^2_{\mathrm{b}}(\bar{\Omega}\times\mathbb{R})} < c$, and the following properties hold for some $\delta > 0$.
\begin{itemize}
\item[{\rm{(1)}}] One has the decomposition
\[
\int_{\Omega}F_a(x, \xi\cdot\psi)\mathrm{d}x =  \mathcal{W}(\xi) + \Upsilon_2(\xi),
\]
where
\begin{equation*}
\mathcal{W}(\rho\cos\theta, \rho\sin\theta) =
 \left\{
\begin{array}{cc}
C\exp\left\{\frac{1}{1-\rho M(\theta)}\right\}\sin\left( \frac{1}{\rho M(\theta)-1} -\theta\right), & \rho M(\theta) >1, \\
0,& \rho M(\theta)\leqslant 1
\end{array}
\right.
\end{equation*}
with constant $C>0$, and $\Upsilon_2(\xi)$ is a smooth function on $B_{\delta}$ satisfying the conditions below.
\begin{itemize}
\item $\Upsilon_2(\xi) = 0$ on $B_{0}$.
\item Let $Q_2(\rho, \theta) := \exp\left\{\frac{1}{\rho M(\theta)-1}\right\}\left(\rho M(\theta)-1\right)^{\frac{1}{2}}\Upsilon_2 (\rho\cos\theta, \rho\sin\theta)$. One has 
\[
	\limsup_{\rho M(\theta)-1\rightarrow 0^{+}}\left\{ |Q_2(\rho, \theta)| + |\partial_{\theta}Q_2(\rho, \theta)| + (\rho M(\theta)-1)^2|\partial_{\rho}Q_2(\rho, \theta)|\right\} < +\infty.
\]
\end{itemize}
\item[{\rm{(2)}}] For any $\varrho > 0$ there exists a positive constant $C(\varrho)$ such that
\begin{equation} \label{estimate:nonlinearity}
	\sup_{x\in\Omega}|f_a(x, \rho\Phi)| + \sup_{x\in\Omega}|f^{\prime}_a(x, \rho\Phi)| \leqslant C(\varrho)\exp\left\{ \frac{1-\varrho}{1- \rho M(\theta)} \right\},
\end{equation}
whenever $1- \rho M(\theta)<0$.
\end{itemize}
\end{lemma}
\begin{proof}
We observe that the integral $\int_{\Omega}F_a(x, \xi\cdot\psi)\mathrm{d}x $ does not involve $\Lambda(\xi)$, and thus, it does not depend on the specific form of the evolution equation in question. Therefore, the proof of this lemma we omit here is the same as that in \cite[Proposition 3.1]{polavcik2002nonconvergent}. 
\end{proof}

Based on these preliminaries above, we are now ready to state our non-convergence result.

\begin{theorem}\label{thm:main1}
Let $a(x)$ and $f_a$ be the functions as in Lemma~\ref{eigenvalue:multiple} and Lemma~\ref{decomposition_2}, respectively. If one chooses $f_0(x, u) = -a(x)u + f_a(x, u)$, then the solution $u(t, \cdot)$ of equation~\eqref{equ:nonconvergence} is bounded in $H^2(\Omega)$ and its $\omega-$limit set is a continuum in $H^1(\Omega)$ homeomorphic to the unit circle $S^1$.
\end{theorem}
\begin{proof}
By lemmas~\ref{cm_reduction}, \ref{decomposition_1}, and~\ref{decomposition_2}, the Lyapunov functional $\Xi(\xi)$ associated with equation~\ref{cmr_ode} can be decomposed into the following form
\[
	\Xi(\xi) = \Upsilon_1(\xi) + \Upsilon_2(\xi) + \mathcal{W}(\xi).
\]
Combining \eqref{estimates_lyapunov_3} with \eqref{estimate:nonlinearity} gives 
\[
|\Upsilon_1(\xi)| + |\partial_{\xi_i}\Upsilon_1(\xi)|\leqslant C(\vartheta, \max\{1, \|f_a\|^2_{C^3_{\mathrm{b}}(\bar{\Omega}\times\mathbb{R})}\})\exp\left\{ \frac{1-\vartheta}{1- \rho M(\theta)} \right\},
\]
which guarantees the following results: 
\begin{itemize}
\item[-] $\Upsilon_1(\xi) = 0$ on $\mathrm{B}_{0}$.
\item[-] Let $Q_1(\rho, \theta) := \exp\left\{\frac{1}{\rho M(\theta)-1}\right\}\left(\rho M(\theta)-1\right)^{\frac{1}{2}}\Upsilon_1 (\rho\cos\theta, \rho\sin\theta)$. One has 
\[
	\limsup_{\rho M(\theta)-1\rightarrow 0^{+}}\left\{ |Q_1(\rho, \theta)| + |\partial_{\theta}Q_1(\rho, \theta)| + (\rho M(\theta)-1)^2|\partial_{\rho}Q_1(\rho, \theta)|\right\} < +\infty.
\]
\end{itemize}
From \cite[Lemma 2.3]{polavcik2002nonconvergent} we know $\Xi(\xi)$ satisfies the three conditions in Proposition~\ref{ode_nonconvergence}. Since we have proved any equilibrium of~\eqref{cmr_ode} is a critical point of $\Xi(\xi)$ in Lemma~\ref{decomposition_1}, by the first and third conditions of Proposition~\ref{ode_nonconvergence} we obtain the set of equilibria of~\eqref{cmr_ode} coincides with $B_{0}$. Thus, all the hypotheses of Proposition~\ref{ode_nonconvergence} have been verified, and hence, we obtain equation~\eqref{cmr_ode} has a bounded solution $\xi(t)$ whose $\omega$-limit set equals a curve $S$ homeomorphic to the unit circle $S^1$. Thanks to Lemma~\ref{cm_reduction}, $u(t, \cdot):=\xi(t)\cdot\psi + \Lambda_1(\xi)$ is a solution of~\eqref{equ:nonconvergence}, and its $\omega$-limit set is
\[
	\omega(u) = \{\xi\cdot\psi + \Lambda_1(\xi): \xi\in S\}\subset H^2(\Omega).
\]
We also know the map $\Lambda(\cdot)\in C^2_{\mathrm{b}}(\mathbb{R}^2\rightarrow\mathcal{D})$ and takes value in $\mathcal{D}\cap E_2$. Therefore, we obtain the conclusion that the solution $u(t, \cdot)$ is bounded in $H^2(\Omega)$ and $\omega(u)$ is homeomorphic to $S^1$. 
\end{proof}

\section{Convergence in a {\L}ojasiewicz-type landscape}
\label{sec:convergence}

In the previous section, it has been shown that smoothness of the interior nonlinearity $f_0$ cannot guarantee the convergence of the solution to points of equilibria. This section will demonstrate that the solution of equation~\eqref{wave_eq2} in a {\L}ojasiewicz-type landscape evolves to some equilibria as time goes to infinity.

\subsection{{\L}ojasiewicz-type landscape}
\label{sec:lojasiewicz}
A functional $E(\cdot)$ is called a {\L}ojasiewicz-type landscape if for each $\varphi\in\mathcal{N}$ there exist numbers $C>0$ and $\sigma >0$, both dependent on $\varphi$, such that for some $\eta\in(0, \frac{1}{2}]$ one has
\begin{equation}\label{lojasiewicz}
	\|- \Delta u +f_0(x, u)  \| + \| \partial_{\nu} u + u\|_{L^{2}(\Gamma)}\geqslant C|E(u) -E(\varphi)|^{1-\eta},
\end{equation}
whenever $u\in H^2(\Omega)$, $\|u-\psi\|_{H^1(\Omega)} < \sigma$. The inequality~\eqref{lojasiewicz} is so-called the {\L}ojasiewicz--Simon inequality, and the number $\eta$ is called {\L}ojasiewicz exponent (cf.~\cite{Hao2006, chill2003lojasiewicz, haraux2003rate}).
As can be seen, the nonlinearity $f_0$ makes $E(u)$ a {\L}ojasiewicz-type landscape. Thus, we refer to this nonlinearity as a {\L}ojasiewicz-type function, and denote by {\L} the space of {\L}ojasiewicz-type functions.

It is necessary to point out that there exists a broad class of {\L}ojasiewicz-type functions. In~\cite[Lemma 3.8]{Hao2006}, the authors have proved that all analytic functions belong to the space {\L}.  
In addition, the function
\begin{equation}\label{example_2}
f_0(x, s) = 
\left
    \{
        \begin{array}{ll}
            \lambda(x)s - se^{-\frac{1}{s^2}}, & s \neq 0, \\
                  0, & s=0,
        \end{array}
\right.
\end{equation}
which has been proposed in Section~\ref{sec:stationary}, is smooth for $s\in\mathbb{R}$ but non-analytic at $s=0$. However, it serves as an explicit example of the {\L}ojasiewicz-type function. Indeed, for the functional $E(\cdot)$ associated with the Robin--Laplace operator and the function~\eqref{example_2}, it is easy to see that the hypotheses in~\cite[Corollary 4.6]{chill2003lojasiewicz} are satisfied. Let $\varphi\neq 0 $ be the critical point of this functional $E(\cdot)$. Then there exist constants $\tilde{C}>0$ and $\tilde{\sigma} > 0$ such that for every $u\in H^2(\Omega)$ with $\|u-\varphi\|_{H^1(\Omega)}<\tilde{\sigma}$
\begin{equation} \label{example_loj1}
	\|E^{\prime}(u)\|_{H^{-1}} \geqslant C|E(u) - E(\varphi)|^{1-\tilde{\eta}},
\end{equation}
where $E^{\prime}(\cdot)$ is the first derivative of $E$. Due to
\begin{equation*}
\begin{split}
\langle E^{\prime}(u), v\rangle &= \int_{\Omega}\nabla u\cdot\nabla v\mathrm{d}x +  \int_{\Gamma}uv\mathrm{d}S + \int_{\Omega}f_0(x, u)v\mathrm{d}x\\
&= \int_{\Omega}\left(-\triangle u + f_0(x, u)\right)v\mathrm{d}x +  \int_{\Gamma}(\partial_{\nu}u+u)v\mathrm{d}S \\
&\leqslant \|-\triangle u + f_0(x, u)\| \|v\|+  \|\partial_{\nu}u+u\|_{L^2(\Gamma)} \|v\|_{L^2(\Gamma)}\\
&\leqslant \left(\|-\triangle u + f_0(x, u)\| +  \|\partial_{\nu}u+u\|_{L^2(\Gamma)} \right)\|v\|_{H^1(\Omega)}
\end{split}
\end{equation*}
for any $v\in H^1(\Omega)$, we get 
\begin{equation}\label{example_loj2}
\|E^{\prime}(u)\|_{H^{-1}} \leqslant \|-\triangle u + f_0(x, u)\| +  \|\partial_{\nu}u+u\|_{L^2(\Gamma)}. 
\end{equation}
Combining\eqref{example_loj1} and \eqref{example_loj2} gives the inequality~\eqref{lojasiewicz}, which means $f_0(x, s)$ given in~\eqref{example_2} is a {\L}ojasiewicz-type function.

\subsection{Geometry of the domain }

To ensure the effectiveness of the restricted boundary damping, it is necessary to impose the following constraint on the shape of the domain $\Omega$.
\begin{itemize}
\item (A4) The domain $\Omega$ is quasi-star-sahped, that is, there exists a $C^3$ vector field 
\[
	\textsl{a}=(\textsl{a}_1, \textsl{a}_2, \textsl{a}_3): \bar{\Omega}\rightarrow\mathbb{R}^3
\]
such that 
	\begin{equation}\label{qstar_sign}
		\textsl{a}(x)\cdot\nu \geqslant 0, \quad x\in\Gamma,
	\end{equation}
	and its Jacobian matrix $\{\frac{\partial \textsl{a}_{i}}{\partial x_j}\}$ is positive definite, that is,
	\begin{equation}\label{qstar_Jacobi}
		\sum_{i, j}\frac{\partial \textsl{a}_{i}(x)}{\partial x_j}\zeta_i\zeta_j \geqslant \varpi|\zeta|^2,\quad x\in\bar{\Omega}, \quad \zeta\in\mathbb{R}^3
	\end{equation}
	for some positive constant $\varpi$.
\end{itemize}
As an illustration, we give some examples of quasi-star-shaped domains: (1) convex domains with smooth boundaries, such as a ball or a cylinder; (2) a star-shaped domain with respect to some point $x_0\in\Omega$, defined as for all $x\in\bar{\Omega}$, $\Omega$ contains the segment $[x_0, x)=\{(1-s)x_0 + sx: 0\leqslant s <1\}$. It is easy to see that a star-shaped domain is not necessarily convex.

We note that the quasi-star-shaped domain possesses the following property, which will be applied to obtain the boundedness of the classical solution to equation~\eqref{wave_eq2}.
\begin{proposition}\label{quasi_star}
If $\Omega$ is a quasi-star-shaped domain, then there exists a $C^3$ vector field 
\[
	\textsl{b}=(\textsl{b}_1, \textsl{b}_2, \textsl{b}_3): \bar{\Omega}\rightarrow\mathbb{R}^3
\]
such that 
\begin{itemize}
\item[{\rm(1)}] for some sufficiently small constant $\varpi_1 >0$ one has
	\begin{equation}\label{qstar_sign_2}
		 \textsl{b}(x)\cdot\nu \geqslant \varpi_1, \quad x\in\Gamma,
	\end{equation}
\item[{\rm(2)}] the Jacobian matrix $\{\frac{\partial \textsl{b}_{i}}{\partial x_j}\}$ is positive definite, in particular,
	\begin{equation}\label{qstar_Jacobi_2}
		\sum_{i, j}\frac{\partial \textsl{b}_{i}(x)}{\partial x_j}\zeta_i\zeta_j \geqslant \varpi_2(\nabla\cdot\textsl{b})|\zeta|^2,\quad x\in\bar{\Omega}, \quad \zeta\in\mathbb{R}^3
	\end{equation}
	for $\varpi_2>0$ small enough,
\item[{\rm(3)}] and its trace satisfies $\nabla\cdot\textsl{b}(x) > \varpi_3:=\varpi_1\frac{|\Gamma_1|}{|\Omega|}$ for $x\in\bar{\Omega}$.
\end{itemize} 
\end{proposition}
\begin{proof}
Let $\textsl{b}(x) := \textsl{a}(x) + \varpi_1\nabla\phi$, where $\varpi_1 >0$ is a constant, and $\phi: \bar{\Omega}\rightarrow\mathbb{R}$ satisfies
\[
\left
    \{
        \begin{array}{rl}
            -\triangle\phi + \frac{|\Gamma_1|}{|\Omega|}= 0 & \textrm{in $\Omega$}, \\
                  \partial_{\nu}\phi = 1 & \textrm{on $\Gamma$}.
        \end{array}
\right.
\]
Then we have 
\begin{equation}\label{qstar_prop}
\begin{split}
\textsl{b}(x)\cdot\nu &= \textsl{a}(x)\cdot\nu + \varpi_1\nu\cdot\nabla\phi,\\
\nabla\cdot\textsl{b}(x) &=\nabla\cdot \textsl{a}(x) + \varpi_1\triangle\phi = \nabla\cdot \textsl{a}(x) + \varpi_1\frac{|\Gamma_1|}{|\Omega|},\\
\frac{\partial\textsl{b}_i(x)}{\partial x_j} &= \frac{\partial\textsl{a}_i(x)}{\partial x_j} + \varpi_1\frac{\partial}{\partial x_j}\left(\frac{\partial\phi(x)}{\partial x_i}\right).
\end{split}
\end{equation}
Combining the first equality of~\eqref{qstar_prop} with~\eqref{qstar_sign} implies~\eqref{qstar_sign_2}. By the positive definite property~\eqref{qstar_Jacobi} of $\{\frac{\partial \textsl{a}_{i}}{\partial x_j}\}$, we have $\nabla\cdot \textsl{a}(x) >0$, and hence, the second equality of~\eqref{qstar_prop} gives
\[
\nabla\cdot\textsl{b}(x) > \varpi_1\frac{|\Gamma_1|}{|\Omega|}, \quad x\in\bar{\Omega}.
\]
Note that the matrix $\{\frac{\partial}{\partial x_j}\left(\frac{\partial\phi(x)}{\partial x_i}\right)\}$ is symmetrical. Let $\lambda_{\mathrm{min}}$ be the minimal eigenvalue of this matrix. From the third equality of~\eqref{qstar_prop} and choosing $\varpi_1\ll \frac{\varpi}{|\lambda_{\mathrm{min}}|}$, it deduces that $\{\frac{\partial \textsl{b}_{i}}{\partial x_j}\}$ is positive definite, and there exists a small constant $\varpi_2>0$ such that 
	\begin{equation*}
	\begin{split}
		\sum_{i, j}\frac{\partial \textsl{b}_{i}(x)}{\partial x_j}\zeta_i\zeta_j &= \sum_{i, j}\frac{\partial \textsl{a}_{i}(x)}{\partial x_j}\zeta_i\zeta_j + \varpi_1\sum_{i, j}\frac{\partial}{\partial x_j}\left(\frac{\partial\phi(x)}{\partial x_i}\right)\zeta_i\zeta_j\\
		&\geqslant \varpi_2(\nabla\cdot\textsl{b})|\zeta|^2,\quad x\in\bar{\Omega}, \quad \zeta\in\mathbb{R}^3,
	\end{split}
	\end{equation*}
which means the inequality~\eqref{qstar_Jacobi_2} holds.
\end{proof}

\begin{remark}
We can replace the definition in {\rm(A4)} of a quasi-star-shaped domain with the one provided in this proposition. Moreover, without loss of generality, we can also assume that there exists a positive constant $\chi$ such that
\begin{equation}\label{qstar_estimate}
|\textsl{b}(x)| + |\nabla\cdot\textsl{b}(x)| + |\nabla(\nabla\cdot\textsl{b}(x))| < \chi, \quad  x\in \Gamma; \quad   |\nabla\cdot\textsl{b}(x)| +|\triangle(\nabla\cdot\textsl{b}(x))|< \chi,  \quad  x\in \Omega
\end{equation}
due to the fact that $\textsl{b}$ is a $C^3$ vector field on $\bar{\Omega}$.
\end{remark}

\subsection{Bounded solution}

This subsection aims to establish some estimates for the classical solution of equation~\eqref{wave_eq2}, which plays a crucial role in the proof of our convergence result.

\begin{lemma}\label{solu:boundedness}
Under the assumptions {\rm(A1)}-{\rm(A4)}, the classical solution $(u, u_t)^\top$ of equation~\eqref{wave_eq2} is bounded in $\mathcal{D}$, and for some $T>0$ large enough, there exists a positive constant $C$, dependent only on $\|(u_0, u_1)^\top\|_{\mathcal{D}}$, such that
\begin{equation}\label{important_inequal}
	\|u_t\|  + \|u_{tt}\| \leqslant C\{\|u_t\|_{L^2(\Gamma)}  + \|u_{tt}\|_{L^2(\Gamma)}\}, \quad t>T.
\end{equation}
\end{lemma}
\begin{proof}
In Lemma~\ref{wellposedness} of Section~\ref{sec:evolution}, we have obtained $(u, u_t)^\top \in C^1([0, \infty); \mathcal{H})$. In fact, we can also derive that $(u, u_t)^\top$ is bounded in $\mathcal{H}$ in the proof of well-posedness. First, we recall the similar steps in \cite[Page 1912-1913]{CEL2002} to demonstrate the boundedness of $(u, u_t)^\top$ in $\mathcal{H}$. By (A1) and the Sobolev embedding theorem, we can see that
\[
\int_{\Omega}F_0(x, u)\mathrm{d}x=\int_{\Omega}\int_{0}^{u}f_0(x, s)\mathrm{d}s\mathrm{d}x\leqslant C_1\int_{\Omega}\left(u+u^4\right) \mathrm{d}x\leqslant C_2\|u\|^2_{H^1},
\]
which implies that 
\begin{equation}\label{bound_1}
\mathcal{E}(u)\leqslant C_3\|(u, u_t)^\top\|^2_{\mathcal{H}}.
\end{equation}
From (A2), we can deduce that for every $c_1>0$ there exist constants $N_0:=N_0(c_1)>0$ such that 
\begin{equation}\label{bound_2}
\int_{\Omega}F_0(x, u)\mathrm{d}x\geqslant -\frac{c+c_1}{2}\int_{\Omega}u^2 + C(\Omega, f_0)
\end{equation}
with $C(\Omega, f_0) =|\Omega| \min\limits_{|s|\leqslant N_0}F_0(x, s)$. The estimate~\eqref{bound_2} leads to
\begin{equation}\label{bound_3}
\begin{split}
\mathcal{E}(u) =& \frac{1}{2}\|u_t\|^2+\frac{1}{2}\|u\|^2_{H^1} + \int_{\Omega}F_0(x, u)\mathrm{d}x\\
\geqslant&\frac{1}{2}\|u_t\|^2 + \frac{1}{2}\left(1-\frac{c+c_1}{\mu_0}\right)\|u\|_{H^1}^2 + C(\Omega, f_0)\\
\geqslant&C_4\|(u, u_t)^\top\|^2_{\mathcal{H}} + C(\Omega, f_0),
\end{split}
\end{equation}
where $c_1$ is chose to be small enough and satisfies $c+c_1 < \mu_0$.
Thus, using~\eqref{bound_1} and~\eqref{bound_3}, along with the energy identity, we obtain the bounded estimate
\begin{equation*}
\begin{split}
\|(u, u_t)^\top\|^2_{\mathcal{H}}\leqslant &C_5\left\{\mathcal{E}(u) - C(\Omega, f_0)\right\}\\
\leqslant &C_6\left\{\mathcal{E}(u(0)) - C(\Omega, f_0)\right\}\\
\leqslant &C_7\left\{\|(u_0, u_1)^\top\|^2_{\mathcal{H}} - C(\Omega, f_0)\right\}.
\end{split}
\end{equation*}
Let $v=u_t$. Then we have $(v, v_t)^\top \in C([0, \infty); \mathcal{H})$, and the estimates
\begin{equation}\label{bound_inH}
\|u\|^2_{H^1(\Omega)} \leqslant C_8, \quad \|v\|^2 \leqslant C_9.
\end{equation}
Additionally, $v$ satisfies the following equation
\begin{equation}\label{eq:v_wave}
\left
    \{
        \begin{array}{l}
	v_{tt} - \Delta v +f^{\prime}_0(x, u)v = 0, \quad (t, x)\in\mathbb{R}^{+} \times \Omega,\\
	\partial_{\nu}v + v +  g^{\prime}(v)v_t = 0,  \quad (t, x)\in \mathbb{R}^{+} \times \Gamma,\\
	v|_{t=0} = u_{1}(x),  v_{t}|_{t=0} = \triangle u_{0} - f_0(x, u_0), \quad  x\in \Omega. 
        \end{array}
\right.
\end{equation}

Multiplying~\eqref{eq:v_wave} by $v_t$ and integrating over $\Omega$, we get
\begin{equation*}
    \int_{\Omega}v_{tt}v_t - \int_{\Omega}\Delta v v_t +\int_{\Omega}f^{\prime}_0(x, u)vv_t=0,
\end{equation*}
which implies
\begin{equation} \label{L2-1}
    \frac{\mathrm{d}}{\mathrm{d}t}\left\{\int_{\Omega}v_{t}^2+\int_{\Omega}|\nabla v|^2+\int_{\Gamma}v^2 +\int_{\Omega}f^{\prime}(x, u)v^2\right\}=\int_{\Omega}f_0^{\prime\prime}(x, u)v^3-2\int_{\Gamma} g^{\prime}(v)v_t^2
\end{equation}	
by integrating by parts and using the boundary condition. Here, it is necessary to point out that, for simplicity, we omit the notation $\mathrm{d}x$ and $\mathrm{d}S$ in the integrals throughout the proof.

Multiplying~\eqref{eq:v_wave} by $2\textsl{b}\cdot\nabla v$ and and integrating by parts over $\Omega$, we have
\begin{equation}\label{L2-2}
\begin{split}
    0=&2\int_{\Omega}v_{tt}(\textsl{b}\cdot\nabla v) - 2\int_{\Omega}\Delta v(\textsl{b}\cdot\nabla v) +2\int_{\Omega}f^{\prime}_0(x, u)v(\textsl{b}\cdot\nabla v)\\
      =&2\frac{\mathrm{d}}{\mathrm{d}t}\left\{\int_{\Omega}v_{t}(\textsl{b}\cdot\nabla v)\right\}-\int_{\Gamma}(\textsl{b}\cdot\nu) v_t^2+\int_{\Omega}(\nabla\cdot\textsl{b}) v_t^2-\int_{\Gamma}\left\{2(\nu\cdot\nabla v)(\textsl{b}\cdot\nabla v)-(\textsl{b}\cdot\nu)|\nabla v|^2\right\}\\
      & -\int_{\Omega}(\nabla\cdot\textsl{b})|\nabla v|^2+2 \sum_{i, j}\int_{\Omega} \frac{\partial \textsl{b}_i}{\partial x_j}\frac{\partial v}{\partial x_i}\frac{\partial v}{\partial x_j} + 2\int_{\Omega}f^{\prime}_0(x, u)v(\textsl{b}\cdot\nabla v).
\end{split}
\end{equation}
By~\eqref{qstar_Jacobi_2}, one has the estimate
\[
-\int_{\Omega}(\nabla\cdot\textsl{b})|\nabla v|^2 + 2 \sum_{i, j}\int_{\Omega} \frac{\partial \textsl{b}_i}{\partial x_j}\frac{\partial v}{\partial x_i}\frac{\partial v}{\partial x_j} \geqslant (2\varpi_2-1)\int_{\Omega}(\nabla\cdot\textsl{b})|\nabla v|^2.
\]
Then from~\eqref{L2-2} we can deduce that
\begin{equation}\label{L2-3}
\begin{split}
    -2\frac{\mathrm{d}}{\mathrm{d}t}\left\{\int_{\Omega}v_{t}(\textsl{b}\cdot\nabla v)\right\}\geqslant&\int_{\Omega}(\nabla\cdot\textsl{b}) v_t^2 -\int_{\Gamma}\left\{2(\nu\cdot\nabla v)(\textsl{b}\cdot\nabla v)+(\textsl{b}\cdot\nu)\left(v_t^2-|\nabla v|^2\right)\right\}\\
      & +(2\varpi_2-1)\int_{\Omega}(\nabla\cdot\textsl{b})|\nabla v|^2 + 2\int_{\Omega}f^{\prime}_0(x, u)v(\textsl{b}\cdot\nabla v).
\end{split}
\end{equation}

Similarly, multiplying~\eqref{eq:v_wave} by $(\nabla\cdot \textsl{b})v$ and integrating by parts over $\Omega$ we obtain
\begin{equation*}
\begin{split}
    0=& \int_{\Omega}v(\nabla\cdot\textsl{b} ) v_{tt}- \int_{\Omega}v(\nabla\cdot\textsl{b} ) \Delta v +\int_{\Omega}f^{\prime}_0(x, u)v^2(\nabla\cdot\textsl{b} )\\
    =&\frac{\mathrm{d}}{\mathrm{d}t}\int_{\Omega}v_{t}v(\nabla\cdot\textsl{b} )-\int_{\Omega}(\nabla\cdot\textsl{b} )v_{t}^2 -\int_{\Gamma}v\partial_{\nu}v(\nabla\cdot \textsl{b} )+\int_{\Omega}(\nabla\cdot \textsl{b} )|\nabla v|^2\\
    &+\frac{1}{2}\int_{\Omega}\nabla v^2 \cdot \nabla(\nabla\cdot \textsl{b} )+\int_{\Omega}f^{\prime}_0(x, u)v^2(\nabla\cdot\textsl{b} )\\
    =&\frac{\mathrm{d}}{\mathrm{d}t}\int_{\Omega}v_{t}v(\nabla\cdot\textsl{b} )-\int_{\Omega}(\nabla\cdot\textsl{b} )v_{t}^2 + \int_{\Gamma}(\nabla\cdot \textsl{b} )v(v+ g^{\prime}(v)v_t)+\int_{\Omega}(\nabla\cdot \textsl{b} )|\nabla v|^2\\
    &+\frac{1}{2}\int_{\Gamma} v^2 (\nu\cdot \nabla)(\nabla\cdot \textsl{b} ) - \frac{1}{2}\int_{\Omega} v^2\triangle(\nabla\cdot \textsl{b} ) +\int_{\Omega}f^{\prime}_0(x, u)v^2(\nabla\cdot\textsl{b} )
\end{split}
\end{equation*}
from which it follows 
\begin{equation}\label{L2-4}
\begin{split}
    &-\frac{\mathrm{d}}{\mathrm{d}t}\left\{\int_{\Omega}v_{t}v(\nabla\cdot\textsl{b} )+\int_{\Gamma} (\nabla\cdot \textsl{b} )vg(v)\right\}\\
    &= -\int_{\Omega}(\nabla\cdot\textsl{b} )v_{t}^2 + \int_{\Gamma}(\nabla\cdot \textsl{b} )v^2 - \int_{\Gamma} (\nabla\cdot \textsl{b} )v_{t}g(v)+\int_{\Omega}(\nabla\cdot \textsl{b} )|\nabla v|^2 \\
    &\quad +\frac{1}{2}\int_{\Gamma} v^2 (\nu\cdot \nabla)(\nabla\cdot \textsl{b} ) - \frac{1}{2}\int_{\Omega} v^2\triangle(\nabla\cdot \textsl{b} ) +\int_{\Omega}f^{\prime}_0(x, u)v^2(\nabla\cdot\textsl{b} ).
\end{split}
\end{equation}

Multiplying \eqref{L2-4} by $1-\varsigma$, $0<\varsigma<\min\{1, 2\varpi_2\}$ and combining it with~\eqref{L2-3} gives
\begin{equation}\label{L2-5}
\begin{split}
& \frac{\mathrm{d}}{\mathrm{d}t}\left\{2\int_{\Omega}v_{t}(\textsl{b}\cdot\nabla v) + (1-\varsigma)\int_{\Omega}v_{t}v(\nabla\cdot\textsl{b} ) + (1-\varsigma)\int_{\Gamma} (\nabla\cdot \textsl{b} )vg(v)\right\} \\
&+\varsigma\varpi_3\int_{\Omega}v_{t}^2+(1-\varsigma)\varpi_3\int_{\Gamma}v^2+(2\varpi_2-\varsigma)\varpi_3\int_{\Omega}|\nabla v|^2 \\
&\leqslant \frac{\mathrm{d}}{\mathrm{d}t}\left\{2\int_{\Omega}v_{t}(\textsl{b}\cdot\nabla v) + (1-\varsigma)\int_{\Omega}v_{t}v(\nabla\cdot\textsl{b} ) + (1-\varsigma)\int_{\Gamma} (\nabla\cdot \textsl{b} )vg(v)\right\} \\
&+\varsigma\int_{\Omega}(\nabla\cdot\textsl{b} )v_{t}^2+(1-\varsigma)\int_{\Gamma}(\nabla\cdot \textsl{b} )v^2+(2\varpi_2-\varsigma)\int_{\Omega}(\nabla\cdot\textsl{b})|\nabla v|^2 \\
 &\leqslant  (1-\varsigma)\int_{\Gamma} (\nabla\cdot \textsl{b} )v_{t}g(v) -\frac{1-\varsigma}{2}\int_{\Gamma} v^2 (\nu\cdot \nabla)(\nabla\cdot \textsl{b} )+ \frac{1-\varsigma}{2}\int_{\Omega} v^2\triangle(\nabla\cdot \textsl{b} ) \\
 &\quad+\int_{\Gamma}\left\{2(\nu\cdot\nabla v)(\textsl{b}\cdot\nabla v)+(\textsl{b}\cdot\nu)\left(v_t^2-|\nabla v|^2\right)\right\}-2\int_{\Omega}f^{\prime}_0(x, u)v(\textsl{b}\cdot\nabla v)\\
 &\quad - (1-\varsigma)\int_{\Omega}f^{\prime}_0(x, u)v^2(\nabla\cdot\textsl{b} ).
\end{split}
\end{equation}
Here, the first inequality is derived by using $\nabla\cdot\textsl{b} > \varpi_3$.

Multiplying \eqref{L2-5} by $\frac{1}{N_1}$ and adding it by~\eqref{L2-1} yields
\begin{equation}\label{L2-6}
\frac{\mathrm{d}}{\mathrm{d}t}\mathcal{J}_0(t) + \mathcal{J}_2(t) \leqslant \sum_{i=3}^7\mathcal{J}_i(t)
\end{equation}
with $\mathcal{J}_0(t):=\int_{\Omega}v_{t}^2+\int_{\Omega}|\nabla v|^2+\int_{\Gamma}v^2+\int_{\Omega}f^{\prime}(x, u)v^2 + \mathcal{J}_{1}(t)$, 
\begin{equation*}
\begin{split}
	\mathcal{J}_1(t):=&\frac{2}{N_1}\int_{\Omega}v_{t}(\textsl{b}\cdot\nabla v) + \frac{1-\varsigma}{N_1}\int_{\Omega}v_{t}v(\nabla\cdot\textsl{b} ) + \frac{1-\varsigma}{N_1}\int_{\Gamma} (\nabla\cdot \textsl{b} )vg(v),\\
	\mathcal{J}_2(t):=&\frac{\varsigma\varpi_3}{N_1}\int_{\Omega}v_{t}^2+\frac{(1-\varsigma)\varpi_3}{N_1}\int_{\Gamma}v^2+\frac{(2\varpi_2-\varsigma)\varpi_3}{N_1}\int_{\Omega}|\nabla v|^2,\\
	\mathcal{J}_3(t):=&\frac{1-\varsigma}{2N_1}\left(\int_{\Gamma}2 (\nabla\cdot \textsl{b} )v_{t}g(v) -\int_{\Gamma} v^2 (\nu\cdot \nabla)(\nabla\cdot \textsl{b} )+ \int_{\Omega} v^2\triangle(\nabla\cdot \textsl{b} )\right)\\
 & + \frac{1}{N_1}\int_{\Gamma}(\textsl{b}\cdot\nu)v_t^2-2\int_{\Gamma} g^{\prime}(v)v_t^2,\\
 	\mathcal{J}_4(t):=&\frac{1}{N_1}\int_{\Gamma}\left\{2(\nu\cdot\nabla v)(\textsl{b}\cdot\nabla v)-(\textsl{b}\cdot\nu)|\nabla v|^2\right\},\quad \mathcal{J}_5(t):=-\frac{2}{N_1}\int_{\Omega}f^{\prime}_0(x, u)v(\textsl{b}\cdot\nabla v),\\
	\mathcal{J}_6(t):=&- \frac{1-\varsigma}{N_1}\int_{\Omega}f^{\prime}_0(x, u)v^2(\nabla\cdot\textsl{b} ),\quad \mathcal{J}_7(t):=\int_{\Omega}f_0^{\prime\prime}(x, u)v^3.
\end{split}
\end{equation*}
We are about to estimate $\mathcal{J}_i$, $i=3, 4, \cdots, 7$, one by one. For $\mathcal{J}_3$, we can deduce from~(A3),~\eqref{qstar_estimate},~\eqref{bound_inH}, and Young's inequality that
\begin{equation*}
\begin{split}
\mathcal{J}_3(t)\leqslant& \frac{(1-\varsigma)\chi}{2N_1}\int_{\Gamma}|g(v)|^2 -(2m_1-\frac{\chi}{N_1}-\frac{(1-\varsigma)\chi}{2N_1})\int_{\Gamma} v_t^2+ \frac{(1-\varsigma)\chi}{2N_1}C_9.
\end{split}
\end{equation*}
By (A3), \eqref{qstar_sign_2}, \eqref{qstar_estimate} and Young's inequality, we have
\begin{equation*}
\begin{split}
\mathcal{J}_4(t)\leqslant&\frac{1}{N_1}\int_{\Gamma}\left\{-2( v + g^{\prime}(v)v_t)(\textsl{b}\cdot\nabla v) - \varpi_2|\nabla v|^2\right\}\\
\leqslant&\frac{1}{N_1}\int_{\Gamma}\left\{2\chi\left( |v||\nabla v| + m_{2}|v_t||\nabla v| \right)- \varpi_2|\nabla v|^2\right\}\\
\leqslant&\frac{1}{N_1}\int_{\Gamma}\left\{2\chi\left( \frac{N_2|v|^2}{2}+\frac{|\nabla v|^2}{2N_2} + m_{2}\frac{N_2|v_t|^2}{2}+m_{2}\frac{|\nabla v|^2}{2N_2} \right)- \varpi_2|\nabla v|^2\right\}\\
\leqslant&\frac{1}{N_1}\int_{\Gamma}\left\{\chi N_2\left( |v|^2+ m_{2}|v_t|^2 \right)- \left(\varpi_2 -\frac{\chi}{N_2}(1+m_2) \right)|\nabla v|^2\right\}
\end{split}
\end{equation*}
with $N_2=N_1^{\frac{1}{2}}$.
Thanks to (A1), \eqref{qstar_estimate}, \eqref{bound_inH}, Young's inequality, H\"older's inequality and the Sobolev embedding theorem, we have
\begin{equation*}
\begin{split}
\mathcal{J}_5(t)\leqslant& \frac{2\chi C_{10}}{N_1}\int_{\Omega}(1+u^2)|v||\nabla v|\\
\leqslant& \frac{\chi C_{10}}{N_1}\int_{\Omega}\left\{N_3v^2 + \frac{1}{N_3}|\nabla v|^2\right\} +\frac{2\chi C_{10}}{N_1}\left(\int_{\Omega}u^6\right)^{\frac{1}{3}}\left(\int_{\Omega}|v|^6\right)^{\frac{1}{6}}\left(\int_{\Omega}|\nabla v|^2\right)^{\frac{1}{2}}\\
\leqslant& \frac{\chi C_{10}}{N_1}\int_{\Omega}\left\{N_3v^2 + \frac{1}{N_3}|\nabla v|^2\right\} +\frac{2\chi C_{11}}{N_1}\left(\int_{\Omega}|\nabla u|^2\right)\left(\int_{\Omega}|\nabla v|^2\right)\\
\leqslant& \frac{\chi C_9 C_{10} N_3}{N_1} +\frac{\chi C_8 C_{11} N_4}{N_1} + \frac{\chi}{N_1}\left(\frac{C_{10}}{N_3}+\frac{C_{11}}{N_4}\right)\int_{\Omega}|\nabla v|^2
\end{split}
\end{equation*}
for some $N_3$, $N_4>0$,
and
\begin{equation*}
\begin{split}
\mathcal{J}_6(t)\leqslant & \frac{C_{12}(1-\varsigma)\chi }{N_1}\int_{\Omega}(1+u^2)v^2\\
\leqslant&\frac{ C_{12}(1-\varsigma)\chi}{N_1}\left\{\int_{\Omega}v^2+ \left(\int_{\Omega}u^6\right)^{\frac{1}{3}}\left(\int_{\Omega}v^3\right)^{\frac{2}{3}}\right\}\\
\leqslant&\frac{ C_{13}(1-\varsigma)\chi}{N_1}\left\{C_9+ C_8\left(\int_{\Omega}v^3\right)^{\frac{2}{3}}\right\}.
\end{split}
\end{equation*}
Moreover, by the Gagliardo--Nirenberg inequality (cf. \cite[Page 11]{nirenberg1959elliptic})
\[
\|v\|_{L^3(\Omega)}\leqslant C_{14}\left(\|\nabla v\|^{\frac{1}{2}}\| v\|^{\frac{1}{2}}+\|v\|\right),
\]
we get
\begin{equation*}
\begin{split}
\mathcal{J}_6(t)\leqslant &\frac{ C_{15}(1-\varsigma)\chi}{N_1}\left\{\left[1+ \left(\frac{N_5}{2}+1\right)C_8\right]C_9 +\frac{C_8}{2N_5} \int_{\Omega}|\nabla v|^2\right\}.
\end{split}
\end{equation*}
From (A1), \eqref{bound_inH}, H\"older's inequality, the Sobolev embedding theorem and the Gagliardo--Nirenberg inequalities
\[
\|v\|_{L^3(\Omega)}\leqslant C_{14}\left(\|\nabla v\|^{\frac{1}{2}}\| v\|^{\frac{1}{2}}+\|v\|\right), \quad \|v\|_{L^4(\Omega)}\leqslant C_{16}\left(\|\nabla v\|^{\frac{1}{4}}\| v\|^{\frac{3}{4}}+\|v\|\right),
\]
it follows
\begin{equation*}
\begin{split}
\mathcal{J}_7(t)\leqslant &C\int_{\Omega}(1+ |u|)|v|^3 \leqslant \frac{C}{2}\int_{\Omega}\left(|v|^3+ |u|^2|v|^2 + |v|^4\right)\\
\leqslant &\frac{C}{2}\left\{\int_{\Omega}|v|^3 +\left( \int_{\Omega}u^6\right)^{\frac{1}{3}}\left( \int_{\Omega}v^3\right)^{\frac{2}{3}}+\int_{\Omega}v^4\right\}\\
\leqslant &C_{17}\Bigg\{\left(\int_{\Omega}|\nabla v|^2\int_{\Omega}| v|^2\right)^{\frac{3}{4}} +\left(\int_{\Omega}|v|^2\right)^{\frac{3}{2}}+\left(\int_{\Omega}|\nabla v|^2\int_{\Omega}| v|^2\right)^{\frac{1}{2}} +\int_{\Omega}|v|^2\\
&+\left(\int_{\Omega}|\nabla v|^2\right)^{\frac{1}{2}}\left(\int_{\Omega}| v|^2\right)^{\frac{3}{2}} +\left(\int_{\Omega}|v|^2\right)^{2}\Bigg\}\\
\leqslant &\frac{C_{17}}{N_6}\int_{\Omega}|\nabla v|^2+C_{18}.
\end{split}
\end{equation*}
By these estimates of $\mathcal{J}_i$, $i=3, 4, \cdots, 7$, we can see that
\begin{equation}\label{L2-7}
\begin{split}
 \sum_{i=3}^7\mathcal{J}_i(t)\leqslant &\frac{\chi}{N_1}\int_{\Gamma}\left\{N_2 v^2+\frac{1-\varsigma}{2}|g(v)|^2\right\} -\left(2m_1-\frac{\chi}{N_1}-\frac{(1-\varsigma)\chi}{2N_1}- \frac{ N_2\chi m_2}{N_1}\right)\int_{\Gamma} v_t^2\\
&+ \frac{\chi C_9}{N_1}\left( \frac{1-\varsigma}{2}+C_{10} N_3 + C_{15}(1-\varsigma)\right)-\frac{1}{N_1} \left(\varpi_2 -\frac{\chi}{N_2}(1+m_2) \right)\int_{\Gamma}|\nabla v|^2\\
& + \left[\frac{\chi}{N_1}\left(\frac{C_{10}}{N_3}+\frac{C_{11}}{N_4} + \frac{C_8 C_{15}(1-\varsigma)}{2N_5}\right) + \frac{C_{17}}{N_6}\right]\int_{\Omega}|\nabla v|^2\\
&+\frac{\chi C_8 }{N_1}\left(\frac{C_{15}(1-\varsigma)C_9(N_5+2)}{2} + C_{11} N_4\right)+C_{18}.
\end{split}
\end{equation}
Choosing $N_1$ to be large enough, we get
\[
2m_1-\frac{\chi}{N_1}-\frac{(1-\varsigma)\chi}{2N_1}- \frac{ N_2\chi m_2}{N_1}>0, \quad \varpi_2 -\frac{\chi}{N_2}(1+m_2)>0.
\]
Set
\[
C_{19}:=\frac{\chi C_9}{N_1}\left( \frac{1-\varsigma}{2}+C_{10} N_3 + C_{15}(1-\varsigma)\right) +\frac{\chi C_8 }{N_1}\left(\frac{C_{15}(1-\varsigma)C_9(N_5+2)}{2} + C_{11} N_4\right) +C_{18}.
\]
Then from~\eqref{L2-7} one has
 \begin{equation}\label{L2-8}
 \sum_{i=3}^7\mathcal{J}_i(t)\leqslant \frac{C_{20}}{\sqrt{N_1}}\int_{\Gamma}\left\{v^2+|g(v)|^2\right\} +C_{12} + \mathcal{J}_{8}(t)
 \end{equation}
 with
 \begin{equation*}
\mathcal{J}_{8}(t) : =\left[\frac{\chi}{N_1}\left(\frac{C_{10}}{N_3}+\frac{C_{11}}{N_4} + \frac{C_8 C_{15}(1-\varsigma)}{2N_5}\right) + \frac{C_{17}}{N_6}\right]\int_{\Omega}|\nabla v|^2.
\end{equation*}
Combing~\eqref{L2-6} and ~\eqref{L2-8} yields
\begin{equation}\label{L2-6_2}
\frac{\mathrm{d}}{\mathrm{d}t}\mathcal{J}_0(t) + \mathcal{J}_2(t)-\mathcal{J}_{8}(t)\leqslant \frac{C_{20}}{\sqrt{N_1}}\int_{\Gamma}\left\{v^2+|g(v)|^2\right\} +C_{19}.
\end{equation}

Let
\[
\frac{C_{21}}{N_1}:=\frac{(2\varpi_2-\varsigma)\varpi_3}{N_1}-\left[\frac{\chi}{N_1}\left(\frac{C_{10}}{N_3}+\frac{C_{11}}{N_4} + \frac{C_8 C_{15}(1-\varsigma)}{2N_5}\right) + \frac{C_{17}}{N_6}\right] .
\]
We have the following equality
\begin{equation}\label{L2-9}
\mathcal{J}_2(t)-\mathcal{J}_{8}(t)=\frac{\varsigma\varpi_3}{N_1}\int_{\Omega}v_{t}^2+\frac{(1-\varsigma)\varpi_3}{N_1}\int_{\Gamma}v^2+\frac{C_{21}}{N_1}\int_{\Omega}|\nabla v|^2.
\end{equation}
Note that
\begin{equation}\label{L2-10}
\begin{split}
\int_{\Omega}f^{\prime}(x, u)v^2 + \mathcal{J}_{1}(t)\leqslant &\left[1+ \left(\frac{N_5}{2}+1\right)C_8\right]C_9 +\left(\frac{C_8}{2N_5} + \frac{\chi}{N_1}\right)\int_{\Omega}|\nabla v|^2\\
&+ \frac{(3-\varsigma)\chi}{2N_1}\int_{\Omega}v_{t}^2 + \frac{(1-\varsigma)\chi}{2N_1}\int_{\Omega}|v|^2 + \frac{(1-\varsigma)\chi}{N_1}\int_{\Gamma} vg(v).
\end{split}
\end{equation}
Choosing $\ell_1< \min\{\varsigma\varpi_3, (1-\varsigma)\varpi_3, C_{21}\}$, from~\eqref{bound_inH}, ~\eqref{L2-9} and~\eqref{L2-10} we can deduce that
\begin{equation}\label{L2-11}
\begin{split}
&\mathcal{J}_2(t)-\mathcal{J}_{8}(t)\\
&=\mathcal{J}_2(t)-\mathcal{J}_{8}(t)+ \frac{\ell_1}{N_1}\left(\int_{\Omega}f^{\prime}(x, u)v^2 + \mathcal{J}_{1}(t)\right)-\frac{\ell_1}{N_1}\left(\int_{\Omega}f^{\prime}(x, u)v^2 + \mathcal{J}_{1}(t)\right)\\
&\geqslant\left(\frac{\varsigma\varpi_3}{N_1}-\frac{(3-\varsigma)\chi\ell_1}{2N^2_1}\right)\int_{\Omega}v_{t}^2 +\frac{(1-\varsigma)\varpi_3}{N_1}\int_{\Gamma}v^2+\left(\frac{C_{21}}{N_1}-\frac{\ell_1}{N_1}\left(\frac{C_8}{2N_5} + \frac{\chi}{N_1}\right)\right)\int_{\Omega}|\nabla v|^2\\
&\quad+ \frac{\ell_1}{N_1}\left(\int_{\Omega}f^{\prime}(x, u)v^2 + \mathcal{J}_{1}(t)\right) -\frac{(1-\varsigma)\chi}{N_1}\int_{\Gamma} vg(v)- C_{22}
\end{split}
\end{equation}
with
\[
C_{22}:=\left[1+ \left(\frac{N_5}{2}+1\right)C_8+\frac{(1-\varsigma)\chi}{2N_1}\right]C_9.
\]
By setting $N_3$, $N_4$, $N_5$, $N_6$ sufficiently large, and choosing $\ell_1$ small enough, one has
\[
\frac{\varsigma\varpi_3}{N_1}-\frac{(3-\varsigma)\chi\ell_1}{2N^2_1}> \frac{\ell_1}{N_1},\quad \frac{C_{21}}{N_1}-\frac{\ell_1}{N_1}\left(\frac{C_8}{2N_5} + \frac{\chi}{N_1}\right)> \frac{\ell_1}{N_1}.
\]
Then by~\eqref{L2-11} we get
\begin{equation} \label{L2-12}
\mathcal{J}_2(t)-\mathcal{J}_{8}(t) \geqslant \frac{\ell_1}{N_1}\mathcal{J}_0(t)-\frac{(1-\varsigma)\chi}{N_1}\int_{\Gamma} vg(v)- C_{22}.
\end{equation}

Since (A3) implies that
\[
vg(v)\geqslant C_{23}\left\{v^2+|g(v)|^2\right\},
\]
then from~\eqref{L2-6_2} and~\eqref{L2-12} we can deduce that
\begin{equation} \label{L2-13}
\frac{\mathrm{d}}{\mathrm{d}t}\mathcal{J}_0(t) +  \frac{\ell_1}{N_1}\mathcal{J}_{0}(t)\leqslant \frac{C_{24}}{\sqrt{N_1}}\int_{\Gamma}vg(v) +C_{25}.
\end{equation}
Due to the energy identity, we have
\[
	\sup_{t\geqslant 0}\int_{t}^{t+1}\left(\int_{\Gamma}v g(v)\right) = \sup_{t\geqslant 0}\left\{\mathcal{E}(t)- \mathcal{E}(t+1)\right\} \leqslant C_{26}.
\]
Then it follows form a Gronwall-type inequality (see \cite[Lemma 2.2]{grasselli2004asymptotic}) that $\mathcal{J}_0(t)$ is bounded, that is, there exists a constant $C_{27}>0$, dependent on $\|(u_0, u_1)^{\top}\|_{\mathcal{D}}$, such that $\mathcal{J}_0(t)\leqslant C_{27}$.

Based on the previous estimates, we notice that there exists a positive constant $\ell_2$ such that
\begin{equation*}
	\|v_t\|^2 + \|\nabla v\|^2 + \|v\|_{L^2({\Gamma})} \leqslant \ell_2\mathcal{J}_0(t)\leqslant \ell_2 C_{27},
\end{equation*}
that is,
\begin{equation}\label{L2-14}
	\|u_{tt}\|^2 + \|\nabla u_t\|^2 + \|u_t\|_{L^2({\Gamma})} \leqslant \ell_2\mathcal{J}_0(t)\leqslant \ell_2 C_{27}.
\end{equation}
Let us consider the following elliptic boundary value problem
\begin{equation*}
\left
    \{
        \begin{array}{ll}
	\Delta u=u_{tt}-f_0(x, u) &\textrm{in $ \Omega$,} \\
	\frac{\partial u}{\partial n}+u=-g(u_t) &\textrm{on $ \Gamma$.}
        \end{array}
\right.
\end{equation*}
By the regularity theory for elliptic problem and~\eqref{L2-14} we obtain
\[
\|u\|_{H^2}\leqslant C_{21}\left\{\|u_{tt}+ \|f_0(x, u)\| +\|g(u_t)\|\right\}.
\]
Therefore, we complete the proof of the boundedness of the solution $\|(u, u_t)^{\top}$ in $\mathcal{D}$.  

To prove the estimate~\eqref{important_inequal}, we employ the method of contradiction. Assume that $\{v_{l}\}$ is a sequence of solutions to equation~\eqref{eq:v_wave} satisfying
	\begin{equation}\label{L2-15}
		\lim_{l\rightarrow +\infty}\frac{\| v_l\|+\| v_{lt}\|}{\|v_{lt}\|^{2}_{L^2(\Gamma)}+\|v_{l}\|^{2}_{L^2(\Gamma)}}=+\infty.
	\end{equation}
Denote $\mathcal{S} = (0, T) \times \Gamma$ and $\mathcal{O} = (0, T)\times \Omega$. From \eqref{L2-14} it deduces that for $T>0$ one has
\begin{itemize} 
\item[-] $v_l \rightarrow v$ weakly in $H^{1}(\mathcal{O})$ and weak-$\ast$ in $L^{\infty}([0, T]; H^{1}(\Omega))$; 
\item[-] $v_l \rightarrow v$ strongly in $L^{2}(\mathcal{S})\cap L^{2}(\mathcal{O})$; 
\item[-] $v_{lt} \rightarrow v_t$ weakly in $L^{2}(\mathcal{O})$.
\end{itemize}
Furthermore, by an Aubin's type compactness lemma as in \cite[Page 65]{simon1986compact}, we also have
	\begin{itemize}
		\item[-] $v_l \rightarrow v$ strongly in $L^{\infty}([0, T]; H^{1-\sigma}(\tau))$ for some $0 < \tau <1 $.
	\end{itemize}

We will argue by contradiction from the following two cases:
\begin{itemize}
\item
Case 1. Assume that $v\neq 0$.
	It follows from~\eqref{eq:v_wave} and~\eqref{L2-15} that
	\begin{itemize}
		\item $v_l$, $v_{lt} \rightarrow 0$ strongly in $L^{\infty}([0, T]; L^{2}(\Gamma))$; 
		\item $f^{\prime}_0(x, u_l)v_l \rightarrow f^{\prime}_0(x, u)v$ strongly in $L^{\infty}([0, T]; L^{2}(\Omega))$.
	\end{itemize}
	Passing to the limit in the equation~\eqref{eq:v_wave} gives
	\begin{equation}
	\left
		\{
        			\begin{array}{ll}
            			v_{tt} - \Delta v +f^{\prime}_0(x, u)v= 0 &\textrm{in $\mathcal{O}$,}  \\
                	 		\partial_{\nu}v = v = v_t= 0 &\textrm{on $\mathcal{S}$}.
       			 \end{array}
	\right.
	\end{equation}
Thanks to $v\in L^{\infty}([0, T]; H^{1-\tau}(\Omega))$, we get $v\in L^{6}(\mathcal{O})$ and $f^{\prime}_0(x, u)\in L^{3}(\mathcal{O})$. Hence, for $T$ large enough, by Theorem 1 in~\cite{ruiz1992unique} we obtain $v= 0$, which contradicts our assumption.
\item
Case 2. Assume that $v=0$.
Let $V_l := \| v_l \|_{L^{2}(\mathcal{O})}$. In terms of the fact that $\| v_l \| \rightarrow 0$ strongly, it is clear that $V_l$ approach to zero as $l$ goes to infinity. 
By set $\tilde{v}_{l} := \frac{v_l}{V_l }$, we can see that $\| \tilde{v}_l \|_{L^{2}(\mathcal{O})}= 1$, and $\tilde{v}_l$ satisfies
\begin{equation}
\left
    \{
        \begin{array}{ll}
            \tilde{v_l}_{tt} - \Delta \tilde{v_l} +f^{\prime}_0(x, u)\tilde{v_l}= 0 & \textrm{in $\mathcal{O}$,}  \\
                 \partial_{\nu}\tilde{v_l} + \tilde{v_l} +\tilde{v_l}_t= 0 &\textrm{on $\mathcal{S}$.}
        \end{array}
\right. \label{L2-16}
\end{equation}
With the same way to prove~\eqref{L2-14}, we can also obtain
\begin{equation}
	\|\tilde{v_{lt}}\| +\|\nabla \tilde{v_l}\| +\|\tilde{v_l}\|_{L^2(\Gamma)}^{2} \leq C_{28}, \label{L2-17}
\end{equation}
which implies
	\begin{itemize}
		\item $\tilde{v}_l \rightarrow \tilde{v}$ weakly in $H^{1}(\mathcal{O})$ and weak-$\ast$ in $L^{\infty}([0, T]; H^{1}(\Omega))$;
		\item $\tilde{v}_l \rightarrow \tilde{v}$ strongly in $L^{2}(\mathcal{S})\cap L^{2}(\mathcal{O})$;
		\item $\tilde{v}_{lt} \rightarrow \tilde{v}_t$ weakly in $L^{2}(\mathcal{O})$.
	\end{itemize}
Here, we know $\| \tilde{v} \|_{L^{2}(\mathcal{O})}=1$. By the Aubin's compactness lemma as above, we have
\begin{itemize}
	\item $\tilde{v}_l \rightarrow \tilde{v}$ strongly in $L^{\infty}([0, T]; H^{1-\tilde{\tau}}(\Omega))$ for some $0 < \tilde{\tau} <1 $.
\end{itemize}
Also from \eqref{L2-15}, we get
	\begin{equation}
		\lim_{l\rightarrow +\infty}\frac{\| \tilde{v}_l\|+\| \tilde{v}_{lt}\|}{\|\tilde{v}_{lt}\|^{2}_{L^2(\Gamma)}+\|\tilde{v}_{l}\|^{2}_{L^2(\Gamma)}}=+\infty.  \label{L2-18}
	\end{equation}
	Since $\| \tilde{v}_l\|_{L^2(\Omega)}^{2}$ is almost everywhere bounded by~\eqref{L2-17}, then by~\eqref{L2-18} we have
	\begin{itemize}
		\item $\tilde{v}_l$, $\tilde{v}_{lt} \rightarrow 0$ strongly in $L^{\infty}([0, T]; L^{2}(\Gamma))$.
	\end{itemize}
Let $l $ goes to infinity, \eqref{L2-16}) becomes the following form
\begin{equation*}
\left
    \{
        \begin{array}{ll}
            \tilde{v}_{tt} - \Delta \tilde{v} +f'(x, u)\tilde{v}= 0 &\textrm{in $\mathcal{O}$,}  \\
                 \partial_{\nu}\tilde{v} =\tilde{v} =\tilde{v}_t= 0 &\textrm{on $\mathcal{S}$.}
        \end{array}
\right.
\end{equation*}
Similarly, we get $\tilde{v}=0$ which contradicts $\| \tilde{v} \|_{L^{2}(\mathcal{O})} =1$. 
\end{itemize}
Thus, the proof of~\eqref{important_inequal} is completed. 
\end{proof}

\subsection{Convergence to equilibria}

Now, we present the statement of our convergence result and provide a detailed proof of the following theorem.

\begin{theorem}\label{thm:main2}
Suppose that the assumptions {\rm(A1)}-{\rm(A4)} hold. If $f_0$ is a {\L}ojasiewicz-type function, the classical solution $(u, u_t)^\top$ of~\eqref{wave_eq2} converges to $(\varphi, 0)^\top$ in $\mathcal{H}$, where $\varphi$ belongs to $\mathcal{N}$.
\end{theorem}
\begin{proof}
We define the $\omega$-limit set of solution $(u_0, u_1)^\top\in\mathcal{H}$ by
\[
	\omega(u_0, u_1) = \{ (\varphi, \phi)^\top\in \mathcal{H}: \textrm{$\exists$ $t_{n}\rightarrow \infty$, such that $(u(t_{n}, x),  u_{t}(t_{n}, x))\xrightarrow{\mathcal{H}}(\varphi, \phi)$}\}.
\]
From~\cite[Theorem 5.2]{CEL2002} it follows that for any $(u_{0}, u_{1})^\top \in \mathcal{D}$, $\omega(u_{0}, u_{1})$ is a nonempty compact connected subset of $\mathcal{H}$, and also invariant with respect to the nonlinear semigroup defined by the solution. Moreover, $\omega(u_{0}, u_{1})$ consists of equilibria, and every element in $\omega(u_{0}, u_{1})$ has the form $(\varphi, 0)^\top$, where $\varphi$ is a solution to problem~\eqref{eq_stationary}, and
\begin{equation}
	\|u_t\|\rightarrow 0, \quad t\rightarrow \infty. \label{thm5_1}
\end{equation}

The following part of proof is about to show that $u$ converges to $\varphi$ in $H^1(\Omega)$, which consists of two steps.

\textbf{Step 1}. Let $\epsilon_{i}$, $i=1$, $2$, $3$, $4$, be small positive real numbers satisfying $\epsilon_1>\epsilon_3>\epsilon_2>\epsilon_4$. For the classical solution $(u, u_t)^\top\in\mathcal{D}$, we introduce an auxiliary function as follow
\begin{equation*}
\begin{split}
	\mathcal{V}(t) =& \mathcal{E}(u) + \epsilon_1\left\{\|\nabla u_t\|^{2} +\|-\Delta u+f_0(x, u)\|^{2}+\int_{\Omega}f_0^{\prime}(x, u)|u_t|^2 \mathrm{d}x\right\}+\epsilon_{2}(-\Delta u+f_0(x, u), u_t)\\
	&-\epsilon_{3}\left\{(\nabla u_t, \nabla u)+\int_{\Omega}f_0(x, u)u_{t}\mathrm{d}x+\int_{\Gamma}u u_{t}\mathrm{d}S\right\} -\epsilon_3\int_{\Gamma}u_t g(u_{t})\mathrm{d}S,
\end{split}
\end{equation*}
which is well defined for all $t\geqslant 0$. By doing some calculation, we obtain
\begin{equation}\label{thm5_2}
\begin{split}
	\frac{\mathrm{d}}{\mathrm{d}t}\mathcal{V}(t) =& -\int_{\Gamma}u_t g(u_{t})\mathrm{d}S - 2\epsilon_1\|u_{tt}\|^2_{L^2(\Gamma)}+  \epsilon_2\|u_{tt}\|^2- \epsilon_2\|-\triangle u + f_0(x, u)\|^2 \\
	&- (\epsilon_3-\epsilon_2) \|u_{t}\|^2_{L^2(\Gamma)}- 2\epsilon_1\int_{\Gamma}u_{tt}^2 g^{\prime}(u_{t})\mathrm{d}S-(\epsilon_3-\epsilon_2) \int_{\Gamma}u_{t}g^{\prime}(u_{t})u_{tt}\mathrm{d}S\\
	&-  (\epsilon_3-\epsilon_2) \|\nabla u_{t}\|^2_{L^2(\Gamma)} - (\epsilon_3-\epsilon_2)\int_{\Omega}f_0^{\prime}(x, u)|u_t|^2 \mathrm{d}x+\epsilon_1\int_{\Omega}f_0^{\prime\prime}(x, u)u_t|u_t|^2 \mathrm{d}x.
\end{split}
\end{equation}
By (A3) and Young's inequality, we get
\begin{equation}\label{thm5_3}
\begin{split}
\int_{\Gamma}u_{tt}^2 g^{\prime}(u_{t})\mathrm{d}S\geqslant m_1\|u_{tt}\|^2_{L^2(\Gamma)}&, \quad \int_{\Gamma}u_t g(u_{t})\mathrm{d}S\geqslant C(m_1, m_2) \left(\|u_{t}\|^2_{L^2(\Gamma)}+\|g(u_{t})\|^2_{L^2(\Gamma)}\right)\\
 -\int_{\Gamma}u_{t}g^{\prime}(u_{t})u_{tt}\mathrm{d}S \leqslant &\frac{m_2}{2}(\|u_{t}\|^2_{L^2(\Gamma)} + \|u_{tt}\|^2_{L^2(\Gamma)}).
\end{split}
\end{equation}
From Lemma~\ref{solu:boundedness} it follows
\begin{equation}\label{thm5_4}
	-\int_{\Omega}f_0^{\prime}(x, u)|u_t|^2 \mathrm{d}x\leqslant C_1\|u_t\|^2,\quad -\int_{\Omega}f_0^{\prime\prime}(x, u)u_t|u_t|^2 \mathrm{d}x\leqslant C_2\|u_t\|^3_{L^3(\Omega)}.
\end{equation}
Moreover, applying the Gagliardo--Nirenberg inequality and H\"older's inequality, we derive the following estimate from the second inequality
\begin{equation}\label{thm5_5}
\|u_t\|^3_{L^3(\Omega)} \mathrm{d}x \leqslant \epsilon_{4}\|\nabla u_t\|^2 +C(\epsilon_{4}) \|u_t\|^6+C_3\|u_t\|^3.
\end{equation}
Combing~\eqref{thm5_3}, ~\eqref{thm5_4} and ~\eqref{thm5_5}, we return to~\eqref{thm5_2} and conclude
\begin{equation}\label{thm5_6}
\begin{split}
	\frac{\mathrm{d}}{\mathrm{d}t}\mathcal{V}(t) \leqslant & -\left(C(m_1, m_2)- \frac{(\epsilon_3-\epsilon_2)(m_2-2)}{2}\right)\|u_{t}\|^2_{L^2(\Gamma)} - C(m_1, m_2)\|g(u_{t})\|^2_{L^2(\Gamma)}\\
	&- 2\epsilon_1\left( m_1 + 1-\frac{(\epsilon_3-\epsilon_2)m_2}{2}\right)\|u_{tt}\|^2_{L^2(\Gamma)}+  \epsilon_2\|u_{tt}\|^2\\
	& + \left\{C_1(\epsilon_3-\epsilon_2)+ C_2  C(\epsilon_{4}) \epsilon_1\|u_t\|^4 + C_2  C_3 \epsilon_1\|u_t\|\right\}\|u_t\|^2\\
	&-  (\epsilon_3-\epsilon_2-C_2 \epsilon_{1}\epsilon_{4}) \|\nabla u_{t}\|^2_{L^2(\Gamma)} - \epsilon_2\|-\triangle u + f_0(x, u)\|^2 .
\end{split}
\end{equation}
Using~\eqref{important_inequal}, \eqref{thm5_1}, and choosing $\epsilon_{i}$ small enough, it immediately derives from~\eqref{thm5_6} that there exists $T_1>0$ such that for all $t>T_1$,
\begin{equation}\label{thm5_7}
\begin{split}
	\frac{\mathrm{d}}{\mathrm{d}t}\mathcal{V}(t) \leqslant& -C_4\left(\|u_{t}\|^2 + \|u_{t}\|^2_{L^2(\Gamma)} + \|g(u_{t})\|^2_{L^2(\Gamma)}+  \|\nabla u_{t}\|^2_{L^2(\Gamma)}+\|-\triangle u + f_0(x, u)\|^2\right)\\
	\leqslant& -\frac{C_4}{4}\left(\|u_{t}\| + \|u_{t}\|_{L^2(\Gamma)} +\|g(u_{t})\|_{L^2(\Gamma)}+  \|\nabla u_{t}\|_{L^2(\Gamma)}+\|-\triangle u + f_0(x, u)\|\right)^2.
\end{split}
\end{equation}
Here, $C_4$ and $C_5$ depend on $m_1$, $m_2$ and $\epsilon_{i}$.

\textbf{Step 2}. 
From~\eqref{thm5_7} we know $\mathcal{V}(t)$ is decreasing on $[T_1, \infty)$, which follows that $\mathcal{V}(t)$ has a finite limit as $t\rightarrow\infty$. 
For $(\varphi, 0)^\top\in\omega(u_0, u_1)$, there is a sequence $t_n$, such that
\begin{equation} \label{thm5_8}
u(t_n, x)\rightarrow\varphi(x), \quad t_n\rightarrow\infty 
\end{equation}
in $\mathcal{H}$. Then we have $E(u(t_n))$ converges to $E(\varphi)$ as $n\rightarrow\infty$. Furthermore, thanks to Lemma~\ref{solu:boundedness},~\eqref{thm5_1} and~\eqref{thm5_7}, we can derive $\mathcal{V}(t)\geqslant E(\varphi)$ for any $t>T_1$.

By Lemma~\ref{solu:boundedness} and H\"older's inequality, we get
\begin{equation*}
\begin{split}
0\leqslant \mathcal{V}(t) - E(\varphi) \leqslant &C_6\Big(\|u_t\|^2 + |E(u) - E(\varphi)| + \|\nabla u_t\|^{2} +\|-\Delta u+f_0(x, u)\|^{2}\\
&\qquad + \|u_{t}\|^2_{L^2(\Gamma)}+ \|g(u_{t})\|^2_{L^2(\Gamma)}+\|-\Delta u+f_0(x, u)\| \|u_{t}\|\Big),
\end{split}
\end{equation*}
which implies
\begin{equation*}
\begin{split}
|\mathcal{V}(t) - E(\varphi) |^{1-\eta}\leqslant &C_7\Big(\|u_t\|^{2(1-\eta)} + |E(u) - E(\varphi)|^{(1-\eta)} + \|\nabla u_t\|^{2(1-\eta)} \\
&\qquad +\|-\Delta u+f_0(x, u)\|^{2(1-\eta)} + \|u_{t}\|^{2(1-\eta)}_{L^2(\Gamma)}+ \|g(u_{t})\|^{2(1-\eta)}_{L^2(\Gamma)} \\
&\qquad +\|-\Delta u+f_0(x, u)\|^{(1-\eta)} \|u_{t}\|^{(1-\eta)} \Big).
\end{split}
\end{equation*}
Here, $\eta\in(0, \frac{1}{2}]$ is the {\L}ojasiewicz exponent. Applying Young's inequality to deal with the last term $\|-\Delta u+f_0(x, u)\|^{(1-\eta)} \|u_{t}\|^{(1-\eta)}$, we have
\begin{equation*}
\begin{split}
|\mathcal{V}(t) - E(\varphi) |^{1-\eta}\leqslant &C_8\Big(\|u_t\|^{2(1-\eta)} + |E(u) - E(\varphi)|^{(1-\eta)} + \|\nabla u_t\|^{2(1-\eta)} \\
&\qquad +\|-\Delta u+f_0(x, u)\|^{2(1-\eta)} + \|u_{t}\|^{2(1-\eta)}_{L^2(\Gamma)} + \|g(u_{t})\|^{2(1-\eta)}_{L^2(\Gamma)} \\
&\qquad +\|-\Delta u+f_0(x, u)\| +  \|u_{t}\|^{\frac{1-\eta}{\eta}} \Big).
\end{split}
\end{equation*}
Note that $2(1-\eta) \geqslant 1$ and $\frac{1-\eta}{\eta} \geqslant 1$. By~\eqref{thm5_1} and Lemma~\ref{solu:boundedness}, we obtain
\begin{equation}\label{thm5_9}
\begin{split}
|\mathcal{V}(t) - E(\varphi) |^{1-\eta}\leqslant &C_9\Big(|E(u) - E(\varphi)|^{(1-\eta)} + \|u_t\| + \|\nabla u_t\| + \|u_{t}\|_{L^2(\Gamma)}\\
&\qquad+\|-\Delta u+f_0(x, u)\| + \|g(u_{t})\|_{L^2(\Gamma)} \Big)
\end{split}
\end{equation}
for any $t>T_1$.

\textbf{Step 3}. From~\eqref{thm5_8} we can deduce that for any $\varepsilon>0$ satisfying $\varepsilon < \sigma$, there is an integer $N$ such that if $n \geq N$, one has $t_n>T_1$ and
\begin{equation}
	\| u(t_{n}, \cdot) - \psi(\cdot)\|_{H^1} \leq \frac{\varepsilon}{2}, \label{thm5_10}
\end{equation}
and
\begin{equation}
	[\mathcal{V}(t_n)-E(\varphi)]^{\eta} -[\mathcal{V}(t)-E(\varphi)]^{\eta}\leq \frac{\varepsilon}{2C_{10}}. \label{thm5_10}
\end{equation}

Now, let us define
\begin{eqnarray*}
	t^{\ast}_{n} :=\sup\{t>t_n: \|u(s, \cdot)-\psi(\cdot) \|_{H^1}<\sigma, \forall s\in [t_n, t] \}.
\end{eqnarray*}
By the continuity of the solution in $H^2(\Omega)$, taking~\eqref{thm5_10} into account gives $t^{\ast}_{n} > t_n$ for any $n\geqslant N$. If $t^{\ast}_{n}= \infty$, then for $t>t_n$, the {\L}ojasiewicz-type inequality~\eqref{lojasiewicz} holds. By~\eqref{lojasiewicz}, \eqref{thm5_7} and~\eqref{thm5_9}, we have
 \begin{equation}\label{thm5_11}
-\frac{\mathrm{d}}{\mathrm{d}t}[\mathcal{V}(t) - E(\varphi)]^{\eta}=-\eta [\mathcal{V}(t) - E(\varphi)]^{\eta-1}\frac{\mathrm{d}}{\mathrm{d}t}\mathcal{V}(t)\geqslant C_{10}\|u_t\|.
\end{equation} 
Integrating~\eqref{thm5_11} over $[t_n, t]$, it can be seen that
\[
	\int_{t_n}^{t}\|u_{t}\|ds\leq C_{10}\{[H(t_n)-E(\psi)]^{\theta} -[H(t)-E(\psi)]^{\theta}\}\leq  \frac{\varepsilon}{2},
\]
which implies $u(t)$ converges in $L^2(\Omega)$ as time goes to infinity. From Lemma \ref{solu:boundedness}, we know $u(t)$ is precompact in $H^1$. Then by the uniqueness of the limit we get 
\[
	\lim_{t\rightarrow\infty}\| u(t)-\psi\|_{H^1}=0.
\]
Another case is that for all $n \geq N$, $t_n<t^{\ast}_{n} < \infty$. Similarly, we have
\begin{equation}\label{thm5_12}
	\int_{t_n}^{t^{\ast}_n}\|u_{t}\|ds\leq C_{10}\{ [H(t_n)-E(\psi)]^{\theta} -[H(\hat{t}_n)-E(\psi)]^{\theta}\} \leq  \frac{\varepsilon}{2}.
\end{equation}
Thus, it follows from~\eqref{thm5_10} and~\eqref{thm5_12} that
\begin{equation*}
	\| u(t^{\ast}_n)-\psi\| \leq \| u(t_n)-\psi\|+ \int_{t_n}^{t^{\ast}_n}\|u_{t}\|ds \leq  \varepsilon,
\end{equation*}
which means that $u(t^{\ast}_n)$ converges to $\varphi$ in $L^2(\Omega)$ as $n \rightarrow\infty$. Since $u(t)$ is precompact in $H^1$, then we can deduce that
\[
	\| u(t^{\ast}_n)-\psi\|_{H^1}< \sigma,
\]
which contradicts the definition of $t^{\ast}_n$. In summary, we complete the proof of this theorem.
\end{proof}

As a byproduct of Theorem~\ref{thm:main2}, we have the following corollary. 
\begin{corollary}
Under the assumptions of Theorem~\ref{thm:main2}, the solution $u$ of~\eqref{wave_eq2} converges to an equilibrium $\varphi$ polynomially in $H^1(\Omega)$ if the {\L}ojasiewicz exponent $\eta\in(0, \frac{1}{2})$, and converges exponentially when $\eta =\frac{1}{2}$.
\end{corollary}
\begin{proof}
Let us define the following auxiliary function
\[
	\mathcal{K}(t):=\int_{t}^{\infty}\left\{\|u_s\| + \|\nabla u_s\|+\|u_s\|_{L^{2}(\Gamma)}+\|g(u_s)\|_{L^{2}(\Gamma)}+ \|-\Delta u+f(x, u)\|\right\}\mathrm{d}s.
\]
Applying Theorem~\ref{thm:main2}, there exists $T>0$ large enough such that for all $t\geqslant T$, it follows from the {\L}ojasiewicz-type inequality~\eqref{lojasiewicz} and~\eqref{thm5_10} that
\begin{equation}
\begin{split}
	\mathcal{K}^{\prime}(t)&=-\left(\|u_t\| + \|\nabla u_t\|+\|u_t\|_{L^{2}(\Gamma)}+\|g(u_t)\|_{L^{2}(\Gamma)}+ \|-\Delta u+f(x, u)\| \right)\\
			&\leq -C_1|\mathcal{V}(t)-E(\varphi)|^{1-\eta}. \label{coro_1}
\end{split}
\end{equation}
Based on~\eqref{thm5_7}, \eqref{thm5_9} and~\eqref{thm5_11} in the proof Theorem~\ref{thm:main2}, we can deduce that $\mathcal{K}(t)\leq |\mathcal{V}(t)-E(\psi)|^{\eta}$, that is, 
\[
	\mathcal{K}(t)^{\frac{1-\eta}{\eta} }\leq |\mathcal{V}(t)-E(\varphi)|^{1-\eta}.
\]
Then combining this and~\eqref{coro_1} gives
\[
	\mathcal{K}(t)^{\prime}\leq -C_1\mathcal{K}(t)^{\frac{1-\eta}{\eta} }, \quad t\geqslant T,
\]
which implies
\begin{equation*}
\mathcal{K}(t) \leqslant
\left
    \{
        \begin{array}{ll}
		\left[\mathcal{K}(T)^{-\frac{\eta}{1-2\eta}} + \frac{1-2\eta}{\eta}C_1(t-T)\right]^{-\frac{\eta}{1-2\eta}}, & \eta \in (0, \frac{1}{2}),\\
		\\
		\mathcal{K}(T) e^{-C_1 (t-T)},& \eta=\frac{1}{2}.
        \end{array}
\right.
\end{equation*}
Since we have
\[
	\| u -\psi \|_{H^1}\leq C_2 \mathcal{K}(t), \quad t\geqslant T,
\]
then this corollary is proved. 
\end{proof}

\section{Conclusion}
\label{sec:conclusion}
This paper systematically investigates problem (Q3) and provides an answer to the open question posed in~\cite{rodrigues2022}. We propose some sufficient conditions on the interior nonlinearity that ensure the equilibria set is infinite, and we also construct an explicit example of such nonlinearity. In this scenario, the convergence results established in~\cite{CEL2002, chueshov2004} are not applicable, as the finiteness of the equilibria set is a crucial assumption in those works. However, we offer a new convergence result that does not require this finiteness condition. In this way, our findings complement the results of~\cite{CEL2002, chueshov2004}. Additionally, we find a class of $C^\infty$-functions for which the wave equation~\eqref{wave_eq2} has a bounded solution whose $\omega$-limit set is a continuum of equilibria. As our future research work, it is interesting to consider the convergence problem for the wave equation~\eqref{wave} in a {\L}ojasiewicz-type landscape.

\section*{Acknowledgments}
The authors thank the anonymous referees very much for the helpful suggestions.

\section*{References}

\bibliography{mybibfile}

\end{document}